\documentclass[12pt, a4paper]{article}

\ifdefined\pdfmapfile
  \IfFileExists{pdftex.map}{}{%
    \pdfmapfile{cm.map}
    \pdfmapfile{+symbols.map}
    \pdfmapfile{+latxfont.map}
    \pdfmapfile{+rsfs.map}%
  }
\fi

\usepackage{latexsym}
\usepackage{booktabs}
\usepackage{caption}
\usepackage{mathrsfs}
\usepackage{amsmath}
\usepackage{amsfonts,amsthm,amssymb,mathrsfs}
\usepackage{float}
\usepackage{graphicx}
\usepackage{enumerate}
\usepackage{tikz}
\usepackage{caption}
\usepackage{rotating}
\allowdisplaybreaks[4]
\usepackage{tabularx}
\usepackage{cite}
\newtheorem{thm}{Theorem}[section]
\newtheorem{lem}[thm]{Lemma}

\addtocounter{section}{0}

\theoremstyle{definition}

\newtheorem{remark}{Remark}[section]

\newtheorem{const}{Construction}[section]

\newtheorem{conj}{Conjecture}[section]

\usepackage{url}

\usepackage[unicode={true},colorlinks,
linkcolor=black,
anchorcolor=blue,
citecolor=green]{hyperref}

\usepackage{authblk}

\usepackage{color}

\definecolor{VeryLightBlue}{rgb}{0.9,0.9,1}
\definecolor{LightBlue}{rgb}{0.8,0.8,1}
\definecolor{MidBlue}{rgb}{0.5,0.5,1}
\definecolor{DarkBlue}{rgb}{0,0,0.6}
\definecolor{Blue}{rgb}{0,0,1}
\definecolor{Gold}{rgb}{1,0.843,0}
\definecolor{LightGreen}{rgb}{0.88,1,0.88}
\definecolor{MidGreen}{rgb}{0.6,1,0.6}
\definecolor{DarkGreen}{rgb}{0,0.6,0}
\definecolor{VeryLightYellow}{rgb}{1,1,0.9}
\definecolor{LightYellow}{rgb}{1,1,0.6}
\definecolor{MidYellow}{rgb}{1,1,0.5}
\definecolor{DarkYellow}{rgb}{1,1,0.2}
\definecolor{DarkPurple}{rgb}{.6,0,1}
\definecolor{Red}{rgb}{1,0,0}
\definecolor{VeryLightRed}{rgb}{1,0.9,0.9}
\definecolor{LightRed}{rgb}{1,0.8,0.8}
\definecolor{MidRed}{rgb}{1,0.55,0.55}

\long\def\delete#1{}

\begin{document}

\title{Extremal spectral gap of regular graphs with bounded vertex connectivity}

\author[1]{Yu Wang\thanks{Corresponding author.}}
\author[2]{Sanming Zhou}

\affil[1]{\small College of Mathematics and System Science, Xinjiang University, Urumqi 830017, P. R. China}
\affil[2]{\small School of Mathematics and Statistics, The University of Melbourne, Parkville, VIC 3010, Australia}
 
\date{}

\openup 0.2\jot 
\maketitle

\begingroup
\renewcommand{\thefootnote}{}
\begin{NoHyper}
\footnotetext{E-mail addresses: yuwang980212@163.com (Y. Wang), sanming@unimelb.edu.au (S. Zhou)}
\end{NoHyper}
\endgroup

\begin{abstract}
	
The well-known inequality of Fiedler ensures that the spectral gap $r-\lambda_2(G)$ of a connected non-complete $r$-regular graph $G$ is bounded from above by the vertex connectivity $\kappa(G)$ of $G$. We prove that, for integers $t\geq2$ and $r>2t^2$, every connected $r$-regular graph with vertex connectivity at most $2t$ has spectral gap at most
$\frac{1}{2}(r+t+2-\sqrt{(r-t+2)^2-4t(t-1)})$. We show that this bound is nearly optimal for each fixed $t \ge 2$ and sufficiently large $r > 2t^2$ such that $r+1$ is divisible by $t$. We further prove that if $r>(2t-1)(2t-2)$, then every connected $r$-regular graph whose vertex connectivity is odd and no more than $2t-1$ has spectral gap at most $r-\min\bigl\{\xi(r,t),\nu(r,t)\bigr\}$ for some explicitly given functions $\xi(r,t)$ and $\nu(r,t)$. If $r\geq6t^2$, then this minimum equals
$\xi(r,t)=\frac{r(2r-4t+5)}{2(r-t+2)}$.
In particular, when $\kappa(G)=2t-1$, this gives
$r-\lambda_2(G)<t=(\kappa(G)+1)/2$,
which improves Fiedler's bound by nearly a factor of two. Our upper bounds in their parameter ranges improve two known bounds on the spectral gap of regular graphs. 



\end{abstract}

\section{Introduction}
\label{sec:int}

Throughout this paper, all graphs are finite, simple and undirected. We write $V(G)$ and $E(G)$ for the vertex and edge sets of a graph $G$, respectively. For $S,\ T\subseteq V(G)$, we use $G[S]$ to denote the subgraph of $G$ induced by $S$, $G-S$ the graph obtained from $G$ by deleting $S$, and $[S,T]$ the set of edges of $G$ with one end in $S$ and the other in $T$. We write $\overline{S}=V(G)\setminus S$ and denote by $\overline{G}$ the complement of $G$. For graphs $G$ and $H$ on disjoint vertex sets, let $G\cup H$ denote their union and $G\vee H$ their join, the latter being obtained from $G\cup H$ by adding all edges between $V(G)$ and $V(H)$. For a positive integer $k$, let $kG$ denote the disjoint union of $k$ copies of $G$. The \emph{vertex connectivity} $\kappa(G)$ is the minimum number of vertices whose deletion disconnects $G$ or leaves a single vertex. We denote the edge connectivity and minimum degree of $G$ by $\kappa'(G)$ and $\delta(G)$, respectively. It is well known that $\kappa(G)\leq\kappa'(G)\leq\delta(G)$.

Let $A(G)$ be the adjacency matrix of $G$. The eigenvalues of $A(G)$, denoted by $\lambda_1(G)\geq\lambda_2(G)\geq\cdots\geq\lambda_n(G)$, are called the eigenvalues of $G$, where $n=|V(G)|$. Thus $\lambda_i(G)$ is the $i$th largest eigenvalue of $G$. The eigenvalues of the Laplacian matrix $L(G)=D(G)-A(G)$ of $G$, written $\mu_1(G)\leq\mu_2(G)\leq\cdots\leq\mu_n(G)$, are called the Laplacian eigenvalues of $G$, where $D(G)$ is the diagonal matrix of vertex degrees of $G$. It is known that $\mu_1(G)=0$, and $\mu_2(G)>0$ when $G$ is connected. In 1973, Fiedler~\cite{fiedler01} proved that $\mu_2(G)\leq\kappa(G)$ for every non-complete graph $G$, and he called $\mu_2(G)$ the algebraic connectivity of $G$. If $G$ is $r$-regular, then $\lambda_1(G)=r$ and $\mu_i(G)=r-\lambda_i(G)$ for $1\leq i\leq n$. In this case the algebraic connectivity $\mu_2(G)$ is precisely the \emph{spectral gap} $r-\lambda_2(G)$ of $G$, and Fiedler's inequality becomes $r-\lambda_2(G)\leq\kappa(G)$.

It is well known that the spectral gap governs the expansion of graphs. Specifically, for every connected $r$-regular graph $G$, we have $\frac{r-\lambda_2(G)}{2}\leq h(G)\leq\sqrt{2r(r-\lambda_2(G))}$~\cite{AM,M}, where $h(G)= \min \{|[S,\overline{S}]|/|S|: \emptyset \neq S\subseteq V(G), |S|\leq |V(G)|/2\}$ is the isoperimetric number of $G$. These famous bounds suggest the extremal problem of finding the maximum or minimum spectral gap of graphs in a given family of connected $r$-regular graphs (see \cite[Problems 1.1 and 1.2]{essential}), say, the family of connected $r$-regular graphs with bounded vertex connectivity, the family of connected $r$-regular graphs with bounded edge connectivity, etc. A related problem is to find the best upper or lower bound for the spectral gap (equivalently, the best lower or upper bound for the second largest eigenvalue) of graphs in a given family of regular graphs. In the case of bounded vertex or edge connectivity, this problem was asked by Mohar and O (private communication), hinted in \cite{cioaba10, cioabagu16}, and formally raised by Abiad \textit{et al.} \cite[Question 1.12]{Abiad01}: For an $r$-regular graph $G$ of a given order and for $1 \le t \le r-1$, what is the best upper bound for $\lambda_2(G)$ which guarantees that $\kappa(G)\geq t+1$ or $\kappa'(G)\geq t+1$? The problem of minimising the spectral gap among connected regular graphs with a given order is closely related to the Aldous--Fill conjecture on the maximum relaxation time of random walks (see \cite{Aldous02,Abdi02,Abdi03,deAbreu01,hoory90}). 

In this paper we focus on the best lower bound for the second largest eigenvalue of $r$-regular graphs with bounded vertex connectivity. Before stating our results, let us review briefly related known results. A series of increasingly strong lower bounds on $\lambda_2(G)$ for $r$-regular graphs $G$ with $\kappa'(G)\leq t$ were obtained in~\cite{Abiad01,kirkland,Chandran01,krivelevich,cioaba10,O90}, culminated with the following result due to O \textit{et al.}~\cite{O90}: For integers $r$ and $t$ with $3\leq t\leq r-1$, every connected $r$-regular graph $G$ with $\kappa'(G)\leq t$ satisfies
$$
\lambda_2(G)\geq\begin{cases}
\frac{r-4+\sqrt{(r+4)^2-8t}}{2}, & \text{ if $t$ is odd}\\[2pt]
\frac{r-3+\sqrt{(r+3)^2-8t}}{2}, & \text{ if $t$ is even,}
\end{cases}
$$
and moreover this bound is sharp. The counterpart best lower bound for the essential edge connectivity was recently obtained by the present authors in~\cite{essential}.

In the case of vertex connectivity, the problem is significantly more difficult and largely open, though the case when $\kappa(G) \leq 1$ or $\kappa(G) \leq 2$ has been settled. In \cite[Theorem 1.8]{O10}, O proved that if $G$ is an $r$-regular graph on at least three vertices such that $\kappa(G) \leq 1$, then $\lambda_2(G) \geq \frac{3r}{4}$. In~\cite[Corollary 1.3]{cioabagu16}, Cioab\u{a} and Gu proved that if $G$ is a connected $r$-regular graph with $r \geq 3$ such that $\kappa(G) \leq 1$, then 
$$
\lambda_2(G) \geq 
\begin{cases}
		\frac{r-2+\sqrt{r^2 + 12}}{2}, & \text{ if $r$ is even}\\[2pt]
		\frac{r-2+\sqrt{r^2+8}}{2}, & \text{ if $r$ is odd.}
	\end{cases}
$$
In~\cite[Theorem 1.13]{Abiad01}, Abiad \textit{et al.} proved that if $G$ is an $r$-regular graph with order $n \geq 5$ and $r \geq 3$ such that $\kappa(G) \leq 1$, then $\lambda_2(G) \geq \frac{8n-25}{9n-25}r$. In \cite[Theorem 1.1]{zhang23}, Zhang and Wang obtained a sharp upper bound on $\lambda_2(G)$ for any connected $r$-regular graph $G$ with $r \geq 3$ that guarantees $\kappa(G) \geq 2$, settling the case when $\kappa(G)=1$. Their result is too complicated to be presented here. Settling the case when $\kappa(G)=2$, in \cite[Theorem 3.2]{zhang03}, Zhang proved that if $G$ is a connected $r$-regular graph with $r \geq 4$ such that $\kappa(G) \leq 2$, then $\lambda_2(G) \geq \lambda_2(H_{r, \lfloor r/2 \rfloor})$, and this bound is attained by $H_{r, \lfloor r/2 \rfloor}$, which is the $r$-regular graph with order $2r+2$ obtained from two vertex-disjoint copies $G_1, G_2$ of $K_r$ and two additional vertices $v_1, v_2$ by adding edges between $v_1$ and $\lfloor r/2 \rfloor$ vertices in $G_1$ and $r - \lfloor r/2 \rfloor$ vertices in $G_2$ and then adding edges between $v_2$ and all vertices in $V(G_1) \cup V(G_2)$ not adjacent to $v_1$. In \cite[Theorem 2.2]{zhang03}, Zhang also proved that if $G$ is a connected $r$-regular graph with $r \geq 3$ such that $\kappa(G) \leq t$ for some integer $t$ with $1 \le t \le r-1$, then
\begin{equation}
\label{eq:zhang}
\lambda_2(G) \geq 
\begin{cases}
r-\frac{t(r+2)}{2(r+3-t)}, & \text{ if $1 \le t < \frac{r}{2} +1$}\\[2pt]
r - t, & \text{ if $\frac{r}{2}+1 \le t \le r-1$.}
\end{cases}
\end{equation}
In \cite[Theorem 1.7]{Liu2019}, Liu \textit{et al.} gave a lower bound on $\lambda_2(G)$ for any $r$-regular graph $G$ such that $\kappa(G) \leq t$ for some integer $t$ with $1 \le t \le r-1$. In \cite[Theorem 4.1]{Zhang2020}, Zhang proved that if $G$ is a connected $r$-regular graph with $r \geq 3$ such that $\kappa(G) \leq t$ for some integer $t$ with $1 \le t \le \frac{r+2}{2}$, then 
\begin{equation}
\label{eq:zhang1}
\lambda_2(G) > r - \frac{tr}{2(r+1-t)}.
\end{equation}

Now let us move on to our own results in this paper. We consider the case when the upper bound on $\kappa(G)$ is even or odd separately. For positive integers $r$ and $l$ such that $r \geq 3l - 2$, let
\begin{equation}\label{eq:rho}
\begin{aligned}
\rho(r,l)=\frac{r-l-2+\sqrt{(r-l+2)^2-4l(l-1)}}{2}.
\end{aligned}
\end{equation}
Our first main result handles the case when the upper bound on $\kappa(G)$ is even and relatively small compared with $r$. 

\begin{thm}\label{thm:1.1} 
Let $t\geq2$ and $r>2t^2$ be integers. Then every connected $r$-regular graph $G$ with $\kappa(G)\leq2t$ satisfies $\lambda_2(G)\geq\rho(r,t)$.
\end{thm}

Define
\begin{equation}\label{eq:lmin}
	\lambda_2^{\min}(r,t)=\inf\bigl\{\lambda_2(G):\ G\ \text{is connected and $r$-regular},\ \kappa(G)\leq 2t\bigr\}. 
\end{equation}
Then Theorem \ref{thm:1.1} can be restated as $\lambda_2^{\min}(r,t) \geq \rho(r,t)$ for $r>2t^2 \geq 8$. To understand the sharpness of this bound, we need to derive an upper bound on $\lambda_2^{\min}(r,t)$ and compare it with $\rho(r,t)$. We are able to achieve this under the extra condition that $t$ is a divisor of $r+1$. Define
\begin{equation}\label{eq:M}
M(\alpha, \beta)=
\begin{pmatrix}
r-t^2 & t^2-t & \alpha\\
r-t^2+1 & t^2-t-2 & \beta\\
\frac{r-t^2+1}{t}\alpha  & (t-1)\beta  & 0
\end{pmatrix}
\end{equation}
for real $\alpha$ and $\beta$. As will be seen in Section~\ref{sec:const}, this matrix is diagonally similar to a real symmetric matrix and hence has only real eigenvalues. We may therefore define
\begin{equation}\label{eq:rhostar}
\rho^*(r,t)=\min\Bigl\{\lambda_{\max}\bigl(M(t-2k_1,t+1-2k_2)\bigr):
k_1,k_2\in\mathbb{Z}, 0\leq k_1\leq t, 1\leq k_2\leq t\Bigr\},
\end{equation}
where $\lambda_{\max}$ means the largest eigenvalue of the matrix. The second main result in this paper is as follows. 
 
\begin{thm}\label{thm:1.2}
Let $t\geq2$ and $r>2t^2$ be integers such that $t\mid(r+1)$. Then $\lambda_2^{\min}(r,t)\leq\rho^*(r,t)$.
\end{thm}

We will prove Theorem~\ref{thm:1.2} by constructing a family of graphs $G^*_{r,t}$ such that $\kappa(G^*_{r,t})=2t$ and $\lambda_2(G^*_{r,t})=\rho^*(r,t)$ (see Section \ref{sec:const}). This construction suggests that the upper bound in Theorem~\ref{thm:1.2} may be sharp. 

\begin{remark}\label{rem:rho-gap}
(a) The quantity $\rho^*(r,t)$ is close to $\rho(r,t)$, indicating that the bound $\lambda_2(G)\geq\rho(r,t)$ is nearly optimal. More precisely, for each fixed $t$, as $r\to\infty$ through integers satisfying $t\mid(r+1)$, we will prove after Lemma~\ref{lem:4.2} that
\[
\rho^*(r,t)-\rho(r,t)
=
\begin{cases}
\dfrac{t(t-1)^2}{r^2}+O(r^{-3}),
   & \text{if $t$ is even},\\[6pt]
\dfrac{1}{t}+O(r^{-1}),
   & \text{if $t$ is odd}.
\end{cases}
\]
Thus, the absolute difference tends to zero when $t$ is even and to $1/t$ when $t$ is odd. Since $\rho(r,t)\sim r$, the relative difference tends to zero in both cases.

(b) The bound $\lambda_2(G)\geq\rho(r,t)$ in Theorem~\ref{thm:1.1} is strictly better than the bounds in~\eqref{eq:zhang} and~\eqref{eq:zhang1} throughout its parameter range, though at the cost that we require $t < \sqrt{r/2}$. Indeed, $r>2t^2$ implies $2t<r/2+1$, so after replacing $t$ in \eqref{eq:zhang} and~\eqref{eq:zhang1} by $2t$, their respective right-hand sides are
\[
z_1=r-\frac{t(r+2)}{r+3-2t}
\quad\text{and}\quad
z_2=r-\frac{tr}{r+1-2t}.
\]
Note that $\rho(r,t)$ is the larger root of
\[
\psi_{r,t}(x)=x^2-(r-t-2)x-\bigl(2r-t(t+1)\bigr).
\]
Since $r>2t^2$, we have $r(3t+1)-4t^3+4t^2-3t+3 > 2t^3+6t^2-3t+3 > 0$ and $r(3t-1)-4t^3+3t-1 > 2t^3-2t^2+3t-1 > 0$. Note that $tr(r-2t)>0$. So a direct calculation gives
\begin{align*}
\psi_{r,t}(z_1)
&=-\frac{t\bigl(tr(r-2t)+r(3t+1)-4t^3+4t^2-3t+3\bigr)}{(r-2t+3)^2}<0,\\
\psi_{r,t}(z_2)
&=-\frac{t\bigl(tr(r-2t)+r(3t-1)-4t^3+3t-1\bigr)}{(r-2t+1)^2}<0.
\end{align*}
Since $\rho(r,t)$ is the larger root of $\psi_{r,t}$, it follows that $z_1, z_2 < \rho(r,t)$ as claimed. 
\end{remark}

\begin{conj}\label{conj:1.1}
For every fixed integer $t\geq2$, there exists an integer $R(t)>2t^2$ such that, for every integer $r\geq R(t)$ satisfying $t\mid(r+1)$,
$\lambda_2^{\min}(r,t)=\rho^*(r,t)$. Equivalently, a graph $G^*_{r,t}$ associated with a pair $(k_1^*,k_2^*)$ achieving the minimum in~\eqref{eq:rhostar} is extremal among all connected $r$-regular graphs with $\kappa(G)\leq2t$.
\end{conj}

We next turn to the case when the upper bound on $\kappa(G)$ is odd and relatively small compared with $r$. Define
\begin{equation}\label{eq:xi-nu-def}
\xi(r,l)=\frac{r(2r-4l+5)}{2(r-l+2)},
\quad \nu(r,l)
=\frac12\!\left(\tau(r,l)+\sqrt{\frac{(r+l)(r+1)\,W(r,l)}
{(2l-1)(r-l+2)^2}}\right),
\end{equation}
where 
\begin{equation}\label{eq:tau}
\tau(r,l)=r-l-\frac{(l-1)(r+l)}{r-l+2},\quad
W(r,l)=(2l-1)r^2-(6l^2-9l+1)r+l(2l-3)(4l-5).
\end{equation}
The expression $\nu(r,l)$ is real and well defined throughout the parameter range in the theorem below. Indeed, if $l\geq2$ and $r>(2l-1)(2l-2)$, put $u=l-1$ and $x=r-2l(l-1)>0$. Direct expansion gives
\begin{equation}\label{eq:W-positive}
W(r,l)=(2u+1)x^2+(8u^3+6u^2+u+2)x
 +(u+1)(8u^4+6u^2-2u+1)>0.
\end{equation}
Thus every factor in the radicand in~\eqref{eq:xi-nu-def} is positive. The third result in this paper is as follows. 

\begin{thm}\label{thm:1.3}
Let $t\geq2$ and $r>(2t-1)(2t-2)$ be integers. Let $G$ be a connected $r$-regular graph such that $\kappa(G)$ is odd and satisfies $\kappa(G)\leq 2t-1$. Then
\begin{equation}\label{eq:odd-bound}
\lambda_2(G)\;\geq\;\min\left\{\xi(r,t),\,\nu(r,t)\right\}.
\end{equation}
In particular, if $r\geq6t^2$, then
\begin{equation}\label{eq:odd-bounds}
\lambda_2(G)\;\geq\;\xi(r,t).
\end{equation}
\end{thm}

\begin{remark}
(a) For $r\geq6t^2$, a direct computation gives
\[
\xi(r,t)-(r-t)=\frac{r-2t(t-2)}{2(r-t+2)}>0.
\]
Thus~\eqref{eq:odd-bounds} implies $r-\lambda_2(G)<t$. In particular, if $\kappa(G)=2t-1$, then
$r-\lambda_2(G)<t=\frac{\kappa(G)+1}{2}$. For the family of graphs $G$ with $\kappa(G)\leq2t-1$, Fiedler's inequality gives the uniform bound $r-\lambda_2(G)\leq2t-1$, whereas~\eqref{eq:odd-bounds} gives $r-\lambda_2(G)<t$ under its hypotheses, improving Fiedler's bound by nearly a factor of two.

(b) The bound~\eqref{eq:odd-bounds} is strictly better than \eqref{eq:zhang} and~\eqref{eq:zhang1}, though at the cost that we require $t \leq \sqrt{r/6}$. Indeed, after replacing $t$ in \eqref{eq:zhang} and~\eqref{eq:zhang1} by $2t-1$, their respective right-hand sides are
\[
y_1=r-\frac{(2t-1)(r+2)}{2(r-2t+4)}
\quad\text{and}\quad
y_2=r-\frac{(2t-1)r}{2(r-2t+2)}.
\]
A direct calculation yields
\begin{align*}
\xi(r,t)-y_1
&=\frac{(2t-1)(rt-2t+4)}{2(r-2t+4)(r-t+2)}>0,\\
\xi(r,t)-y_2
&=\frac{rt(2t-1)}{2(r-2t+2)(r-t+2)}>0.
\end{align*}

(c) At present we do not know whether equality in either bound of Theorem~\ref{thm:1.3} can be attained.
\end{remark}

In the next section we will present the algebraic and spectral tools that will be used in later sections. The proofs of Theorems \ref{thm:1.1}, \ref{thm:1.2} and \ref{thm:1.3} will be given in Sections \ref{sec:pf1}, \ref{sec:const} and \ref{sec:pf2}, respectively. 

\section{Preliminaries}
\label{sec:tl}

A matrix $A=(a_{ij})$ of order $n$ is \emph{non-negative} if $a_{ij} \geq 0$ for all $1 \leq i, j \leq n$, and \emph{tridiagonal} if $a_{ij}=0$ whenever $|i-j|>1$. Given a non-negative matrix $M$ of order $n$, its associated digraph has vertex set $\{1,2,\ldots,n\}$, with $ij$ an arc if and only if $M_{ij}>0$. The matrix $M$ is \emph{irreducible} if its associated digraph is strongly connected. If an $n \times n$ matrix $M$ has only real eigenvalues, then we denote its eigenvalues by $\lambda_1(M)\geq\lambda_2(M)\geq\cdots\geq\lambda_n(M)$, and in particular $\lambda_{\max}(M) = \lambda_1(M)$. The following result relates the spectrum of a non-negative tridiagonal matrix with constant row sum to that of a smaller matrix obtained by removing the trivial eigenvalue.

\begin{thm}[\cite{brouwer89,horn90}]\label{thm:2.1}
Let $A$ be an $(n+1) \times (n+1)$ non-negative tridiagonal matrix of the form
\[
A =\begin{pmatrix}
a_0 & b_0 & & & 0 \\
c_1 & a_1 & b_1 & & \\
& \ddots & \ddots & \ddots & \\
& & \ddots & \ddots & b_{n-1} \\
0 & & & c_n & a_n
\end{pmatrix}.
\]
Assume that every row sum of $A$ equals $r$. If the eigenvalues of $A$ are indexed in non-increasing order as $\lambda_1, \lambda_2, \ldots, \lambda_{n+1}$, then the $n \times n$ matrix
\[
\tilde{A} =
\begin{pmatrix}
r - b_0 - c_1 & b_1 & & & 0 \\
c_1 & r - b_1 - c_2 & b_2 & & \\
& c_2 & \ddots & \ddots & \\
& & \ddots & \ddots & b_{n-1} \\
0 & & & c_{n-1} & r - b_{n-1} - c_n
\end{pmatrix}
\]
has eigenvalues $\lambda_2, \lambda_3, \ldots, \lambda_{n+1}$. In particular, $\lambda_1(A) = r$ and $\lambda_2(A) = \lambda_1(\tilde{A})$.
\end{thm}

Let $G$ be a graph and $\mathcal P=\{V_1,\ldots,V_s\}$ a partition of $V(G)$. The corresponding quotient matrix $Q=(Q_{ij})$ is defined by
$Q_{ij}=\frac{|[V_i,V_j]|}{|V_i|}\quad(i\ne j)$ and
$Q_{ii}=\frac{2|E(G[V_i])|}{|V_i|}$. The partition $\mathcal P$ is \emph{equitable} if every vertex in $V_i$ has exactly $Q_{ij}$ neighbours in $V_j$ for all $i,j$.

\begin{thm}[{\cite[Quotient Interlacing Theorem]{brouwer11}}, see also \cite{godsil01}]\label{thm:2.2}
Let $G$ be a graph and let $Q$ be the quotient matrix corresponding to a partition $\mathcal P$ of $V(G)$. Then the eigenvalues of $Q$ interlace those of $G$. If $\mathcal P$ is equitable, then every eigenvalue of $Q$ is an eigenvalue of $G$, and the spectral radius of $Q$ equals that of $G$.
\end{thm}

\begin{lem}\label{lem:2.3}
Let $r$ and $t$ be positive integers. Let $G$ be a connected $r$-regular graph, and let $S$ be a vertex cut of $G$ such that $G-S$ has components $C_1, C_2,\ldots,C_k$.
\begin{itemize}
\item[{\rm(i)}] If $|S| = 2t$ and $r > 2t^2$, then $|C_i| \geq r-2t+1$ for $1\leq i\leq k$.
\item[{\rm(ii)}] If $|S| = 2t-1$ and $r > (2t-1)(2t-2)$, then $|C_i| \geq r-2t+2$ for $1\leq i\leq k$.
\end{itemize}
\end{lem}

\begin{proof}
Consider a component $C_i$ and a vertex $v\in C_i$. Since $C_i$ is a component of $G-S$, every neighbour of $v$ lies in $C_i\cup S$. Thus $v$ has at least $r-|S|$ neighbours in $C_i$, and hence $|C_i|\geq r-|S|+1$. Substituting $|S|=2t$ and $|S|=2t-1$ gives~(i) and~(ii), respectively.
\end{proof}

\noindent \textbf{Convention.} At several points in the next three sections, we will differentiate functions originally defined on discrete variables; when doing so, we regard these variables as real variables.

\section{Proof of Theorem~\ref{thm:1.1}}
\label{sec:pf1}

Throughout this section, put $\rho_l=\rho(r,l)$ and $k=r-l+1$, where $\rho$ is defined in~\eqref{eq:rho}, and write
\begin{equation}\label{eq:eta-l}
 \eta=r-l-\rho_l.
\end{equation}
Then $\eta$ is the smaller root of
\begin{equation}\label{eq:eta-equation}
 z^2-(r-l+2)z+l(l-1)=0.
\end{equation}
Thus $\eta=0$ when $l=1$, whereas $0<\eta<1$ when $l\geq2$, because the polynomial in~\eqref{eq:eta-equation} is positive at $0$ and negative at $1$ under the assumption $r>2l^2$.

\subsection{Minimising the five-part quotient}

\begin{lem}\label{lem:3.1}
Let $l$, $c$, $b$ and $r$ be positive integers such that $1\leq c\leq l$, $b\geq1$ and $r>2l^2$, and let $\eta$ be the smaller root of~\eqref{eq:eta-equation}. Let
\begin{align}
\Xi(c,b)={}&bc^2+2bcl-bcr-bc-bl\eta-bl-c^3-2c^2l+3c^2r+3c^2-2cl^2\notag\\
&+4clr-cl\eta+5cl-2cr^2-4cr-2c+2l^2r+2l^2-2lr-2l,
\label{eq:Xi-def}
\end{align}
and
\begin{align}
N(c)={}&-c^4+(-3l+3r+4)c^3+\bigl(-2\eta l-5l^2+6lr+11l-2r^2-5r-3\bigr)c^2\notag\\
&+\bigl(-2\eta l^2-2\eta l-3l^3+5l^2r+8l^2-4lr-6l\bigr)c-2\eta l^3-2\eta l^2r-2\eta l^2\notag\\
&+4\eta lr+8\eta l-l^3-l^2r+l^2.
\label{eq:N-def}
\end{align}
Then $\Xi(c,b)<0$ and $N(c)<0$.
\end{lem}

\begin{proof}
The coefficient of $b$ in $\Xi(c,b)$ is $c(c+2l-r-1)-l(\eta+1)<0$,
and hence $\Xi(c,b)\leq\Xi(c,1)$. The coefficient of $\eta$ in $\Xi(c,1)$ is $-l(c+1)$, so
$\Xi(c,1)\leq\Xi_0(c,r,l):=\Xi(c,1)|_{\eta=0}$. Put $\gamma=c-1\geq0$, $\sigma=l-c\geq0$, and $r_0=2l^2+1$. Each of the following expansions is obtained by substituting $c=\gamma+1$, $l=\gamma+\sigma+1$ and $r_0=2l^2+1$, and then collecting monomials in $\gamma$ and $\sigma$:
\begin{align*}
\left.\frac{\partial\Xi_0}{\partial r}\right|_{r=r_0}
={}&-8\sigma^2\gamma-6\sigma^2-16\sigma\gamma^2-24\sigma\gamma-10\sigma-8\gamma^3-15\gamma^2-17\gamma-10<0.
\end{align*}
Since $\partial^2\Xi_0/\partial r^2=-4c<0$, $\Xi_0$ decreases for $r\geq r_0$. Moreover,
\begin{align*}
\Xi_0(c,r_0,l)={}&-4(2\gamma+1)\sigma^4
-4(2\gamma+1)(4\gamma+3)\sigma^3-2(24\gamma^3+45\gamma^2+34\gamma+11)\sigma^2\\
&-(32\gamma^4+76\gamma^3+90\gamma^2+61\gamma+20)\sigma-(\gamma+1)(8\gamma^4+14\gamma^3+21\gamma^2\\
&+10\gamma+10)<0.
\end{align*}
This proves $\Xi(c,b)<0$.

For $N(c)$, its coefficient of $\eta$ is $-2l\bigl(c^2+cl+c+l^2+lr+l-2r-4\bigr)<0$ when $l\geq2$. For $l=1$, we have $\eta=0$. Thus $N(c)\leq N_0(c,r,l):=N(c)|_{\eta=0}$. Since $\partial^2N_0/\partial r^2=-4c^2<0$, it suffices to verify that $\partial N_0/\partial r<0$ at $r=r_0$ and that $N_0(c,r_0,l)<0$. With the same $\gamma$ and $\sigma$ as above,
\begin{align*}
\left.\frac{\partial N_0}{\partial r}\right|_{r=r_0}
={}&-8\sigma^2\gamma^2-11\sigma^2\gamma-4\sigma^2
-16\sigma\gamma^3-32\sigma\gamma^2\\
&-22\sigma\gamma-6\sigma-8\gamma^4-18\gamma^3-20\gamma^2-18\gamma-8<0
\end{align*}
and
\begin{align*}
N_0(c,r_0,l)={}&-2\gamma(4\gamma+3)\sigma^4
-\gamma(32\gamma^2+44\gamma+11)\sigma^3\\
&-2(\gamma+1)(24\gamma^3+21\gamma^2+4\gamma+2)\sigma^2\\
&-2(\gamma+1)(16\gamma^4+20\gamma^3+9\gamma^2+2\gamma+2)\sigma\\
&-4(\gamma+1)^2(2\gamma^4+\gamma^3+(\gamma-1)^2)<0.
\end{align*}
Hence $N(c)<0$.
\end{proof}

\begin{lem}\label{lem:3.2}
Let $l$, $t$ and $r$ be positive integers such that $1\leq l\leq t$ and $r>2t^2$. Let $m_1$, $m_2$, $p$ and $q$ be positive integers such that
\begin{equation}\label{eq:coarse-admissible}
 m_1,\ m_2\geq r-2l+1,\;
 p\geq m_1(r-m_1+1),\;
 q\geq m_2(r-m_2+1)\;\ \text{and}\;\
 p+q=2lr.
\end{equation}
Set
\begin{equation}\label{eq:Q1-general}
Q_1=\begin{pmatrix}
r-\dfrac{p}{m_1}-\dfrac{p}{2l} & \dfrac{q}{2l}\\
 \dfrac{p}{2l} & r-\dfrac{q}{m_2}-\dfrac{q}{2l}
\end{pmatrix}.
\end{equation}
If $\lambda_1(Q_1)<\rho_l$, then $m_1=m_2=k=r-l+1$.
\end{lem}

\begin{proof}
Set $y=r-\rho_l=l+\eta$. Since $Q_1$ is diagonally similar to a symmetric matrix, its eigenvalues are real. Hence $\lambda_1(Q_1)<\rho_l$ implies
\begin{equation}\label{eq:D-positive}
 D:=\det(\rho_l I-Q_1)>0.
\end{equation}
We prove that~\eqref{eq:D-positive} cannot hold unless $m_1=m_2=k$.

First suppose that $m_1,\ m_2\geq k$. For fixed $m_1$, $m_2$ and $q=2lr-p$, the determinant $D$ is a concave quadratic function of $p$. Its maximum value is
\begin{equation}\label{eq:Dmax}
\frac{\bigl(2lr+(m_1+m_2)\rho_l\bigr)H}{2m_1m_2(2l+m_1+m_2)},
\end{equation}
where $H=2l^2r+l(m_1+m_2)\rho_l-2m_1m_2(r-\rho_l)$.
Write $m_1=k+a$ and $m_2=k+b$, where $a,\ b\geq0$. Using~\eqref{eq:eta-equation}, we obtain
\begin{equation}\label{eq:H-decomp}
 H=H_0+C(a+b)-2(l+\eta)ab,
\end{equation}
where $H_0=2(r+1)\eta(1-\eta)$ and $C=lr-(2r-l+2)(l+\eta)$.
Moreover,
\begin{equation}\label{eq:H0C}
H_0+C=l(l-r-2)+l\eta-2(r+1)\eta^2<0.
\end{equation}
Indeed, $0\leq\eta<1$ and $r>l-1$. Since $H_0\geq0$, equation~\eqref{eq:H0C} also gives $C<0$. Therefore, if $a+b\geq1$, then~\eqref{eq:H-decomp} yields $H\leq H_0+C<0$.
It follows from~\eqref{eq:Dmax} that $D<0$. Hence, when $m_1,\ m_2\geq k$, condition~\eqref{eq:D-positive} can hold only when $m_1=m_2=k$.

It remains to consider the case in which $m_1<k$ or $m_2<k$. Let $g(x)=x(r-x+1)$. Since $g$ is decreasing on $[r-2l+1,k-1]$, if $m_1<k$, then $p\geq g(m_1)\geq g(k-1)=(r-l)(l+1)>lr$.
The same conclusion holds if $m_2<k$. Hence $m_1$ and $m_2$ cannot both be less than $k$. Moreover, $m_1<k$ and $m_2=k$ would give $p+q\geq (r-l)(l+1)+lk=2lr+(r-2l^2)>2lr$.
By symmetry, the only remaining case is $m_1<k<m_2$. Write
\begin{equation}\label{eq:mixed-orders}
m_1=k-c,\quad m_2=k+b,\quad 1\leq c\leq l,\quad b\geq1.
\end{equation}
Set $p_0=m_1(r-m_1+1)=(k-c)(l+c)$. For fixed $m_2$, the determinant $D=D(p,m_2)$ is a concave function of $p$. Differentiating at $p=p_0$ and using~\eqref{eq:eta-equation} gives
\begin{equation}\label{eq:D-derivative}
\left.\frac{\partial D}{\partial p}\right|_{p=p_0}
 =\frac{\Xi(c,b)}{l(k-c)(k+b)},
\end{equation}
where $\Xi(c,b)$ is defined in~\eqref{eq:Xi-def}. By Lemma~\ref{lem:3.1}, $\Xi(c,b)<0$. Thus, by concavity, $D(p,m_2)$ is strictly decreasing for $p\geq p_0$.

With $p=p_0$, differentiation with respect to $m_2$ gives
$\frac{\partial D(p_0,m_2)}{\partial m_2}
=-\frac{q}{m_2^2}\left(c-\eta+\frac{p_0}{2l}\right)<0$, since
$q>0$, $c-\eta>0$ and $p_0>0$. Hence $D(p,m_2)\leq D(p_0,k+1)$. Finally, substituting $p=p_0$ and $m_2=k+1$, multiplying by the positive denominator $2l(k+1)$, and reducing $\eta^2$ by~\eqref{eq:eta-equation}, we obtain
\begin{equation}\label{eq:mixed-boundary-D}
D(p_0,k+1)=\frac{N(c)}{2l(k+1)},
\end{equation}
where $N(c)$ is defined in~\eqref{eq:N-def}. By Lemma~\ref{lem:3.1}, $N(c)<0$, so~\eqref{eq:mixed-boundary-D} contradicts~\eqref{eq:D-positive}. This rules out the remaining case and completes the proof.
\end{proof}

\begin{lem}\label{lem:3.3}
Let $l$, $t$ and $r$ be positive integers such that $1\leq l\leq t$ and $r>2t^2$, and put $k=r-l+1$. Let $G$ be an $r$-regular graph and $\{V_1, S, V_2\}$ a partition of $V(G)$ such that $S$ is an independent set, $[V_1,V_2]=\emptyset$, $|S|=2l$ and $|V_1|=|V_2|=k$. Then $\lambda_2(G)\geq\rho_l$. Equality can occur only if $|[S,V_1]|=|[S,V_2]|=lr$ and, in each of $V_1$ and $V_2$, exactly $l(l-1)$ vertices have $l+1$ neighbours in $S$, while all remaining vertices have $l$ neighbours in $S$.
\end{lem}

\begin{proof}
For $l=1$, the independence of $S$ gives $|[S,V_1]|+|[S,V_2]|=2r$. Since $|V_i|=r$ and $[V_1,V_2]=\emptyset$, every vertex of $V_i$ has at least one neighbour in $S$. Hence $|[S,V_1]|=|[S,V_2]|=r$, and each $V_i$ induces $K_r$. The quotient matrix of $\{V_1,S,V_2\}$ is
\[
\begin{pmatrix}
r-1&1&0\\[2pt]
r/2&0&r/2\\[2pt]
0&1&r-1
\end{pmatrix},
\]
whose eigenvalues are $r$, $r-1$ and $-1$. Hence the result holds for $l=1$. 

We now assume $l\geq2$. Put $h=l(l-1)$, $p=|[S,V_1]|$ and $q=|[S,V_2]|$. Since $S$ is independent, degree counting gives $p+q=2lr$. Thus, for some $\varepsilon$, we may write $p=lr+\varepsilon$ and $q=lr-\varepsilon$.
Every vertex of $V_i$ has at least $l$ neighbours in $S$, because $|V_i|=k$. Since the minimum possible number of edges from $V_1$ or $V_2$ to $S$ is $lk=lr-h$, we have
\begin{equation}\label{eq:epsilon-range}
|\varepsilon|\leq h.
\end{equation}
Let $e_1=h+\varepsilon$ and $e_2=h-\varepsilon$. Thus $e_i$ is the total number of neighbours in $S$ in excess of $l$, summed over all vertices of $V_i$.

Assume first that $e_1,e_2>0$. For $i=1,2$, let $W_i$ be the set of vertices of $V_i$ having more than $l$ neighbours in $S$, put $u_i=|W_i|$, and set $U_i=V_i\setminus W_i$. Since each vertex of $W_i$ contributes an integer between $1$ and $l$ to $e_i$, we have $1\leq u_i\leq e_i\leq2h=2l(l-1)<k$, where the last inequality follows from $r>2t^2$ and $l\leq t$. Thus both $W_i$ and $U_i$ are nonempty, and in particular $k-u_i>0$. Define $z_i=e_i/u_i$. Then $1\leq z_i\leq l$. Every vertex of $U_i$ has exactly $l$ neighbours in $S$ and therefore $k-1$ neighbours in $V_i$. The quotient matrix corresponding to the partition
$\{U_1,W_1,S,W_2,U_2\}$ is
\begin{equation}\label{eq:T-epsilon}
T=\begin{pmatrix}
k-u_1-1&u_1&l&0&0\\
k-u_1&u_1-z_1-1&l+z_1&0&0\\
\dfrac{k-u_1}{2}&\dfrac{u_1(l+z_1)}{2l}&0&
\dfrac{u_2(l+z_2)}{2l}&\dfrac{k-u_2}{2}\\
0&0&l+z_2&u_2-z_2-1&k-u_2\\
0&0&l&u_2&k-u_2-1
\end{pmatrix}.
\end{equation}
This quotient matrix is diagonally similar to a symmetric matrix. Deleting the row and column corresponding to $S$ leaves a matrix similar to $L_1\oplus L_2$, where
\[
L_i=\begin{pmatrix}
k-u_i-1&u_i\\
k-u_i&u_i-z_i-1
\end{pmatrix}
\]
for $i=1,2$. Each $L_i$ is diagonally similar to a real symmetric matrix and, since $u_i>0$ and $k-u_i>0$, is an irreducible Metzler matrix (that is, its off-diagonal entries are nonnegative). Choose $c>0$ sufficiently large such that $L_i+cI$ is nonnegative. By the Perron--Frobenius theorem and the minimum-row-sum bound,
\[
\lambda_{\max}(L_i)+c
=\lambda_{\max}(L_i+cI)
\geq c+\min\{k-1,k-z_i-1\}
=c+k-z_i-1>c.
\]
Thus $L_i$ has a positive eigenvalue. Since $\det(-I-L_i)=-z_i(k-u_i)<0$, its other eigenvalue is less than $-1$.

Set $a_i=z_i-1\geq0$. Substituting $\rho_l=r-l-\eta$ into the characteristic polynomial of~\eqref{eq:T-epsilon}, and using
$h=\eta(k+1-\eta)$, which follows from~\eqref{eq:eta-equation}, we obtain
\begin{equation}\label{eq:det-refined}
\det(\rho_l I-T)=\frac{(l+\eta)(k-\eta)}{2l}\,B,
\end{equation}
where $B=2\varepsilon^2+(l-\eta)\varepsilon(a_1-a_2)+h(l+\eta)(a_1+a_2)+2\eta l a_1a_2$. By~\eqref{eq:epsilon-range} and $|a_1-a_2|\leq a_1+a_2$,
\begin{align*}
B &\geq 2\varepsilon^2+
\bigl(h(l+\eta)-(l-\eta)|\varepsilon|\bigr)(a_1+a_2)
+2\eta l a_1a_2\\
&\geq 2\varepsilon^2+2\eta h(a_1+a_2)+2\eta l a_1a_2\\
& \geq 0.
\end{align*}
Moreover, equality holds only when
$\varepsilon=0$ and $a_1=a_2=0$. It follows that $z_1=z_2=1$ and $u_1=e_1=u_2=e_2=h$.

Let the eigenvalues of $T$ be $r=\vartheta_1\geq\vartheta_2\geq\cdots\geq\vartheta_5$. If the second largest eigenvalue of $L_1\oplus L_2$ is at least $\rho_l$, interlacing gives $\vartheta_2\geq\rho_l$. Otherwise, interlacing gives $\vartheta_3<\rho_l$. In this case, the last three factors in 
$\det(\rho_l I-T)=\prod_{i=1}^5(\rho_l-\vartheta_i)$ are positive, whereas $\rho_l-r<0$. Equation~\eqref{eq:det-refined} therefore implies $\vartheta_2\geq\rho_l$. By Theorem~\ref{thm:2.2}, the eigenvalues of $T$ interlace those of $G$, and hence $\lambda_2(G)\geq\rho_l$. Furthermore, if $\lambda_2(G)=\rho_l$, then $\lambda_2(G)\geq\vartheta_2\geq\rho_l$ forces $\vartheta_2=\rho_l$. Hence $\det(\rho_l I-T)=0$. Since the prefactor in~\eqref{eq:det-refined} is positive, we must have $B=0$. By the equality condition in the estimate above, this implies $\varepsilon=0$ and $z_1=z_2=1$, and consequently $e_1=e_2=u_1=u_2=h$. Thus, in each $V_i$, exactly $h=l(l-1)$ vertices have $l+1$ neighbours in $S$, while all remaining vertices have $l$ neighbours in $S$.

When $\varepsilon=0$ and $z_1=z_2=1$, the matrix $T$ is invariant under reversal of the five parts. The block corresponding to the symmetric subspace has eigenvalues $r$, $-1$ and $-(l+1)$, whereas the antisymmetric block has characteristic polynomial
$x^2-(r-l-2)x-(2r-l^2-l)$. The larger root of this polynomial is $\rho_l$, and consequently $\lambda_2(T)=\rho_l$.

It remains to consider the case $e_1e_2=0$. By symmetry, suppose that $e_2=0$, so $\varepsilon=h$ and $G[V_2]\cong K_k$. Define $U_1$, $W_1$, $u_1$ and $z_1$ as above. Then $u_1z_1=2h$, and the corresponding four-part quotient matrix is
\[
T_0=\begin{pmatrix}
k-u_1-1 &u_1 &l &0\\
k-u_1 &u_1-z_1-1 &l+z_1 &0\\
\dfrac{k-u_1}{2} &\dfrac{u_1(l+z_1)}{2l} &0 &\dfrac{k}{2}\\
0 &0 &l &k-1
\end{pmatrix}.
\]
Substituting $u_1z_1=2h$, clearing denominators, and reducing powers of $\eta$ by~\eqref{eq:eta-equation}, we obtain
\begin{equation}\label{eq:det-extreme}
\det(\rho_l I-T_0)=\eta(l+\eta)(k-\eta)(l-2+z_1)>0.
\end{equation}
Deleting the row and column corresponding to $S$ leaves $L_1\oplus[k-1]$. Moreover, since $\rho_l=k-1-\eta$, $u_1z_1=2h$ and $h=\eta(k+1-\eta)$, we have $\det(\rho_l I-L_1)=2h-\eta(k+z_1)+\eta^2=\eta(k+2-z_1-\eta)>0$, where the last inequality follows from $z_1\leq l$, $0<\eta<1$ and $r>2l^2$. The argument above shows that $L_1$ has one positive eigenvalue and one eigenvalue less than $-1$. Since $\rho_l>0$, this determinant inequality implies $\lambda_{\max}(L_1)<\rho_l$. Also, $k-1=\rho_l+\eta>\rho_l$. Therefore the second largest eigenvalue of $L_1\oplus[k-1]$ is strictly less than $\rho_l$. Interlacing and the sign in~\eqref{eq:det-extreme} now give $\lambda_2(T_0)>\rho_l$.
Consequently, $\lambda_2(G)\geq\lambda_2(T_0)>\rho_l$, so equality cannot occur in this case. This completes the proof.
\end{proof}

\subsection{Proof of Theorem~\ref{thm:1.1}}
\label{subsec:proof-i}

\begin{proof}[Proof of Theorem~\ref{thm:1.1}]
Let $t \ge 2$ and $r > 2t^2$ be integers. Let $G$ be a connected $r$-regular graph with $\kappa(G)\leq2t$, and let $S_0$ be a minimum vertex cut. If $|S_0|$ is even, put $S=S_0$. If $|S_0|$ is odd, then $|S_0|\leq2t-1<r$, and every component $C$ of $G-S_0$ satisfies $|C|\geq r-|S_0|+1\geq2$.
Choose a vertex from one component of $G-S_0$ and add it to $S_0$. In either case, the resulting set $S$ has even size $|S|=2l\leq2t$, while $G-S$ still has at least two components.

Let $C_1,\ldots,C_s$ be the components of $G-S$, where $s \ge 2$. Choose one component $C_1$ and put $V_1=C_1$ and $V_2=\bigcup_{i=2}^s C_i$.
Set $m_i=|V_i|$, $p=|[S,V_1]|$ and $q=|[S,V_2]|$.
Since $G$ is connected, every component of $G-S$ has a neighbour in $S$, and hence $p,q>0$.
The quotient matrix of the partition $\{V_1,S,V_2\}$ of $V(G)$ is
\[
Q_0=\begin{pmatrix}
r-\dfrac{p}{m_1}&\dfrac{p}{m_1}&0\\[5pt]
\dfrac{p}{2l}&r-\dfrac{p+q}{2l}&\dfrac{q}{2l}\\[5pt] 0&\dfrac{q}{m_2}&r-\dfrac{q}{m_2}
\end{pmatrix}.
\]
Applying Theorem~\ref{thm:2.1} to remove the eigenvalue $r$ yields the matrix $Q_1$ in~\eqref{eq:Q1-general}, while Theorem~\ref{thm:2.2} gives
\begin{equation}\label{eq:lambda-Q1}
 \lambda_2(G)\geq\lambda_1(Q_1).
\end{equation}
Lemma~\ref{lem:2.3}(i) and degree counting give $m_1,m_2\geq r-2l+1$, $p\geq m_1(r-m_1+1)$, $q\geq m_2(r-m_2+1)$ and $p+q\leq2lr$.

\medskip
\noindent\textbf{Claim 1.}
$\lambda_1(Q_1)$ is strictly decreasing in each of $p$ and $q$.

\smallskip
\noindent\emph{Proof.}
Write $\tau=\operatorname{tr}(Q_1)$ and $\Delta=\tau^2-4\det(Q_1)$. Then
$\lambda_1(Q_1)=\frac{\tau+\sqrt\Delta}{2}$.
Differentiating gives $\frac{\partial\tau}{\partial p}=-\frac{m_1+2l}{2lm_1}<0$ and
\[
\left(\frac{\partial\Delta}{\partial p}\right)^2
-4\left(\frac{\partial\tau}{\partial p}\right)^2\Delta
=-\frac{2q^2(2l+m_1+m_2)}{l^3m_1m_2}<0.
\]
If $\partial\Delta/\partial p\leq0$, then the derivative of $(\tau+\sqrt\Delta)/2$ with respect to $p$ is immediately negative. If $\partial\Delta/\partial p>0$, the displayed strict inequality gives $\frac{1}{2\sqrt\Delta}\frac{\partial\Delta}{\partial p}<-\frac{\partial\tau}{\partial p}$, and the derivative is again negative. Thus $\partial\lambda_1(Q_1)/\partial p<0$. The argument for $q$ is identical. \hfill$\square$

\medskip
\noindent\textbf{Claim 2.}
If $\lambda_2(G)<\rho_l$, then $m_1=m_2=k=r-l+1$ and $p+q=2lr$.

\smallskip
\noindent\emph{Proof.}
Assume that $\lambda_2(G)<\rho_l$. By~\eqref{eq:lambda-Q1},
$\lambda_1(Q_1)<\rho_l$. Choose integers $p'\geq p$ and $q'\geq q$ such that
$p'+q'=2lr$, and let $Q_1'$ be obtained from $Q_1$ by replacing $p$ and $q$ with $p'$ and $q'$. Claim~1 gives
$\lambda_1(Q_1')\leq\lambda_1(Q_1)<\rho_l$.
The parameters $m_1$, $m_2$, $p'$ and $q'$ satisfy~\eqref{eq:coarse-admissible}, so Lemma~\ref{lem:3.2} yields $m_1=m_2=k$.

It remains to show that the original values satisfy $p+q=2lr$. Put $w=p+q$. When $m_1=m_2=k$, the larger eigenvalue of $Q_1$ is
\[
\lambda_1(Q_1)=r-\frac{w}{2}\left(\frac1k+\frac1{2l}\right)
 +\frac12\sqrt{\left(\frac1k+\frac1{2l}\right)^2(p-q)^2
 +\frac{pq}{l^2}}.
\]
Since
\[
\left(\frac1k+\frac1{2l}\right)^2(p-q)^2
+\frac{pq}{l^2}-\frac{w^2}{4l^2}
=\left[\left(\frac1k+\frac1{2l}\right)^2-\frac1{4l^2}
 \right](p-q)^2\geq0,
\]
we have
\begin{equation}\label{eq:Q1-total-bound}
\lambda_1(Q_1)\geq r-\frac{w}{2k}.
\end{equation}
If $w\leq2lr-1$, then $\lambda_1(Q_1)\geq r-\frac{2lr-1}{2k}=\left(r-\frac{lr}{k}\right)+\frac1{2k}$.
On the other hand, writing $h=l(l-1)$ and using $h=\eta(k+1-\eta)$, we obtain $\rho_l-\left(r-\frac{lr}{k}\right)=\frac{h}{k}-\eta=\frac{\eta(1-\eta)}{k}<\frac1{2k}$.
Thus $w\leq2lr-1$ would imply $\lambda_1(Q_1)>\rho_l$, a contradiction. Hence $p+q=2lr$. \hfill$\square$

\medskip
By Claim~2, any counterexample with $\lambda_2(G)<\rho_l$ has
$|V_1|=|V_2|=k$ and $p+q=2lr$. The equality $p+q=2lr$ implies that $S$ is an independent set. Lemma~\ref{lem:3.3} then gives
$\lambda_2(G)\geq\rho_l$, a contradiction. Therefore $\lambda_2(G)\geq\rho_l=\rho(r,l)$, where $\rho$ is defined in~\eqref{eq:rho}.
Finally,
\[
\frac{d}{dl}\left(
\frac{r-l-2+\sqrt{(r+2)^2-l(3l+2r)}}{2}\right)
=-\frac12-\frac{r+3l}{2\sqrt{(r+2)^2-l(3l+2r)}}<0.
\]
Since $l\leq t$, it follows from~\eqref{eq:rho} that $\lambda_2(G)\geq\rho(r,t)$.
This proves Theorem~\ref{thm:1.1}.
\end{proof}

\section{Proof of Theorem~\ref{thm:1.2}}
\label{sec:const}

In this section, we first construct a family of connected $r$-regular graphs $G_{r,t}(k_1,k_2)$ for integers $0\leq k_1\leq t$ and $1\leq k_2\leq t$. We then prove that these graphs have vertex connectivity $2t$ and their second largest eigenvalue equals the largest eigenvalue of $M(t-2k_1,t+1-2k_2)$, leading to the proof of Theorem~\ref{thm:1.2}, where $M(\alpha,\beta)$ is as defined in~\eqref{eq:M}.

\subsection{Construction}

\begin{const}\label{const:4.1}
(a) Let $t$ and $r$ be integers satisfying $t\geq2$, $r>2t^2$ and $t\mid(r+1)$, and put
\begin{equation}\label{eq:Frt}
d=r-t^2+1,
\qquad
F_{r,t}=K_d\vee\overline{\tfrac{t^2-t}{2}K_2}.
\end{equation}
Take two vertex-disjoint copies of $F_{r,t}$. In the first copy, denote the vertex set of the $K_d$-part by $U_1$ and the vertex set of the $\overline{\frac{t^2-t}{2}K_2}$-part by $U_2$. In the second copy, denote the vertex sets of the corresponding parts by $U_4$ and $U_3$, respectively. Let $S_1$ and $S_2$ be two disjoint sets of size $t$, disjoint from both copies of $F_{r,t}$, and put $S=S_1\cup S_2$. We take $S$ to be an independent set in the graph being constructed.

Fix integers $k_1$ and $k_2$ with $0\leq k_1\leq t$ and $1\leq k_2\leq t$. For each $i\in\{1,2,3,4\}$ and $j\in\{1,2\}$, add a biregular bipartite graph between $U_i$ and $S_j$. If $d_{ij}$ is the entry in row $U_i$ and column $S_j$ of the following table, then every vertex of $U_i$ has exactly $d_{ij}$ neighbours in $S_j$:
\begin{equation}\label{eq:degree-table}
\begin{array}{c|cc}
 & S_1&S_2\\ \hline
U_1&t-k_1&k_1\\
U_2&t+1-k_2&k_2\\
U_3&k_2&t+1-k_2\\
U_4&k_1&t-k_1
\end{array}
\end{equation}
The degrees from $S_j$ to $U_i$ are determined by edge counting: every vertex of $S_j$ has $|U_i|d_{ij}/t$ neighbours in $U_i$. Since $t\mid d$ and $|U_2|=|U_3|=t(t-1)$, all the prescribed degrees are integers. The bounds on $k_1$ and $k_2$ also ensure that no prescribed degree exceeds the order of the opposite part.

By the Gale--Ryser criterion~\cite{Gale57,Ryser57}, the integrality, degree bounds, and edge-counting identities verified above guarantee that every pair $(U_i,S_j)$ admits the prescribed simple biregular realisation. Choose any such realisations, add no other edges, and denote the resulting graph by $G_{r,t}(k_1,k_2)$. 

(b) Take a pair of integers $(k_1^*,k_2^*)$ with $0\leq k_1^*\leq t$ and $1\leq k_2^* \leq t$ that achieve the minimum in \eqref{eq:rhostar}. Choose a corresponding realisation in the construction above and denote the graph by 
$$
G^*_{r,t}=G_{r,t}(k_1^*,k_2^*). 
$$ 
\end{const}

Figure~\ref{fig1} illustrates the case $(k_1,k_2)=(0,t)$. In this case, every vertex of $U_1$ has all its $t$ neighbours in $S_1$, whereas every vertex of $U_2$ has $t$ of its $t+1$ neighbours in $S_2$.
\begin{figure}[H]
	\centering
	\scalebox{0.8}{\begin{tikzpicture}[x=1.00mm, y=1.00mm, inner xsep=0pt, inner ysep=0pt, outer xsep=0pt, outer ysep=0pt]
			\path[line width=0mm] (59.59,23.09) rectangle +(134.92,71.38);
			\definecolor{L}{rgb}{0,0,0}
			\path[line width=0.15mm, draw=L] (96.24,82.62) [rotate around={180:(96.24,82.62)}] ellipse (16.54mm and 4.65mm);
			\path[line width=0.15mm, draw=L] (157.60,82.42) [rotate around={180:(157.60,82.42)}] ellipse (25.60mm and 4.45mm);
			\definecolor{F}{rgb}{0,0,0}
			\path[line width=0.15mm, draw=L, fill=F] (85.02,82.45) circle (0.50mm);
			\path[line width=0.15mm, draw=L, fill=F] (89.68,82.29) circle (0.50mm);
			\path[line width=0.15mm, draw=L, fill=F] (95.00,82.45) circle (0.30mm);
			\path[line width=0.15mm, draw=L, fill=F] (98.32,82.45) circle (0.30mm);
			\path[line width=0.15mm, draw=L, fill=F] (101.64,82.45) circle (0.30mm);
			\path[line width=0.15mm, draw=L, fill=F] (108.29,82.45) circle (0.50mm);
			\path[line width=0.15mm, draw=L, fill=F] (136.15,82.45) circle (0.50mm);
			\path[line width=0.15mm, draw=L, fill=F] (138.31,82.45) circle (0.30mm);
			\path[line width=0.15mm, draw=L, fill=F] (140.50,82.39) circle (0.30mm);
			\path[line width=0.15mm, draw=L, fill=F] (142.82,82.39) circle (0.30mm);
			\path[line width=0.15mm, draw=L, fill=F] (145.15,82.39) circle (0.50mm);
			\path[line width=0.15mm, draw=L] (165.40,36.02) [rotate around={180:(165.40,36.02)}] ellipse (16.79mm and 4.83mm);
			\path[line width=0.15mm, draw=L, fill=F] (152.76,35.92) circle (0.50mm);
			\path[line width=0.15mm, draw=L, fill=F] (157.42,35.75) circle (0.50mm);
			\path[line width=0.15mm, draw=L, fill=F] (162.74,35.92) circle (0.30mm);
			\path[line width=0.15mm, draw=L, fill=F] (166.06,35.92) circle (0.30mm);
			\path[line width=0.15mm, draw=L, fill=F] (169.38,35.92) circle (0.30mm);
			\path[line width=0.15mm, draw=L, fill=F] (176.03,35.92) circle (0.50mm);
			\draw(87.75,89.77) node[anchor=base west]{\fontsize{9}{11}\selectfont $U_1:K_{r-t^2+1}$};
			\draw(145.13,90.12) node[anchor=base west]{\fontsize{9}{11}\selectfont $U_2:\overline{\frac{t^2-t}{2}K_2}$};
			\path[line width=0.15mm, draw=L] (130.89,35.75) -- (147.98,35.70);
			\path[line width=0.15mm, draw=L, fill=F] (130.89,35.75) -- (132.29,35.40) -- (130.89,35.75) -- (132.29,36.10) -- (130.89,35.75) -- cycle;
			\path[line width=0.15mm, draw=L, fill=F] (147.98,35.70) -- (146.58,36.06) -- (147.98,35.70) -- (146.57,35.36) -- (147.98,35.70) -- cycle;
			\draw(91.86,25.74) node[anchor=base west]{\fontsize{9}{11}\selectfont $U_3:\overline{\frac{t^2-t}{2}K_2}$};
			\definecolor{L}{rgb}{1,0,0}
			\path[line width=0.15mm, draw=L, dash pattern=on 0.15mm off 0.50mm] (76.05,76.64) rectangle +(107.82,12.04);
			\path[line width=0.15mm, draw=L, dash pattern=on 0.15mm off 0.50mm] (75.06,29.69) rectangle +(108.81,12.68);
			\draw(61.59,81.71) node[anchor=base west]{\fontsize{9}{11}\selectfont $r+1-t$};
			\draw(61.59,34.33) node[anchor=base west]{\fontsize{9}{11}\selectfont $r+1-t$};
			\draw(97.70,67.68) node[anchor=base west]{\fontsize{9}{11}\selectfont $S_2:\overline{K_t}$};
			\definecolor{L}{rgb}{0,0,0}
			\path[line width=0.15mm, draw=L] (102.71,59.47) [rotate around={180:(102.71,59.47)}] ellipse (16.54mm and 4.65mm);
			\path[line width=0.15mm, draw=L] (155.10,59.70) [rotate around={180:(155.10,59.70)}] ellipse (16.63mm and 4.42mm);
			\path[line width=0.15mm, draw=L, fill=F] (91.49,59.30) circle (0.50mm);
			\path[line width=0.15mm, draw=L, fill=F] (96.15,59.13) circle (0.50mm);
			\path[line width=0.15mm, draw=L, fill=F] (101.46,59.30) circle (0.30mm);
			\path[line width=0.15mm, draw=L, fill=F] (104.79,59.30) circle (0.30mm);
			\path[line width=0.15mm, draw=L, fill=F] (108.11,59.30) circle (0.30mm);
			\path[line width=0.15mm, draw=L, fill=F] (114.76,59.30) circle (0.50mm);
			\path[line width=0.15mm, draw=L, fill=F] (142.62,59.30) circle (0.50mm);
			\path[line width=0.15mm, draw=L, fill=F] (147.27,59.13) circle (0.50mm);
			\path[line width=0.15mm, draw=L, fill=F] (152.59,59.30) circle (0.30mm);
			\path[line width=0.15mm, draw=L, fill=F] (155.92,59.30) circle (0.30mm);
			\path[line width=0.15mm, draw=L, fill=F] (159.24,59.30) circle (0.30mm);
			\path[line width=0.15mm, draw=L, fill=F] (165.89,59.30) circle (0.50mm);
			\draw(157.95,25.74) node[anchor=base west]{\fontsize{9}{11}\selectfont $U_4:K_{r-t^2+1}$};
			\path[line width=0.15mm, draw=L, fill=F] (149.44,82.45) circle (0.50mm);
			\path[line width=0.15mm, draw=L, fill=F] (151.60,82.45) circle (0.30mm);
			\path[line width=0.15mm, draw=L, fill=F] (153.79,82.39) circle (0.30mm);
			\path[line width=0.15mm, draw=L, fill=F] (156.12,82.39) circle (0.30mm);
			\path[line width=0.15mm, draw=L, fill=F] (158.44,82.39) circle (0.50mm);
			\path[line width=0.15mm, draw=L, fill=F] (169.38,82.45) circle (0.50mm);
			\path[line width=0.15mm, draw=L, fill=F] (161.57,82.45) circle (0.30mm);
			\path[line width=0.15mm, draw=L, fill=F] (163.76,82.39) circle (0.30mm);
			\path[line width=0.15mm, draw=L, fill=F] (166.08,82.39) circle (0.30mm);
			\path[line width=0.15mm, draw=L, fill=F] (178.38,82.39) circle (0.50mm);
			\definecolor{L}{rgb}{0,0,0.545}
			\path[line width=0.15mm, draw=L, dash pattern=on 0.15mm off 0.50mm] (140.83,82.39) ellipse (6.48mm and 2.16mm);
			\path[line width=0.15mm, draw=L, dash pattern=on 0.15mm off 0.50mm] (154.21,82.64) ellipse (6.06mm and 2.24mm);
			\path[line width=0.15mm, draw=L, dash pattern=on 0.15mm off 0.50mm] (173.98,82.64) ellipse (6.73mm and 2.24mm);
			\definecolor{L}{rgb}{0,0,0}
			\path[line width=0.15mm, draw=L, fill=F] (171.54,82.45) circle (0.30mm);
			\path[line width=0.15mm, draw=L, fill=F] (173.73,82.39) circle (0.30mm);
			\path[line width=0.15mm, draw=L, fill=F] (176.05,82.39) circle (0.30mm);
			\path[line width=0.15mm, draw=L] (114.27,82.27) -- (131.36,82.22);
			\path[line width=0.15mm, draw=L, fill=F] (114.27,82.27) -- (115.67,81.92) -- (114.27,82.27) -- (115.67,82.62) -- (114.27,82.27) -- cycle;
			\path[line width=0.15mm, draw=L, fill=F] (131.36,82.22) -- (129.96,82.57) -- (131.36,82.22) -- (129.96,81.87) -- (131.36,82.22) -- cycle;
			\path[line width=0.15mm, draw=L] (104.43,35.90) [rotate around={180:(104.43,35.90)}] ellipse (25.60mm and 4.45mm);
			\path[line width=0.15mm, draw=L, fill=F] (82.99,35.93) circle (0.50mm);
			\path[line width=0.15mm, draw=L, fill=F] (85.15,35.93) circle (0.30mm);
			\path[line width=0.15mm, draw=L, fill=F] (87.34,35.87) circle (0.30mm);
			\path[line width=0.15mm, draw=L, fill=F] (89.66,35.87) circle (0.30mm);
			\path[line width=0.15mm, draw=L, fill=F] (91.99,35.87) circle (0.50mm);
			\path[line width=0.15mm, draw=L, fill=F] (96.28,35.93) circle (0.50mm);
			\path[line width=0.15mm, draw=L, fill=F] (98.44,35.93) circle (0.30mm);
			\path[line width=0.15mm, draw=L, fill=F] (100.63,35.87) circle (0.30mm);
			\path[line width=0.15mm, draw=L, fill=F] (102.95,35.87) circle (0.30mm);
			\path[line width=0.15mm, draw=L, fill=F] (105.28,35.87) circle (0.50mm);
			\path[line width=0.15mm, draw=L, fill=F] (116.22,35.93) circle (0.50mm);
			\path[line width=0.15mm, draw=L, fill=F] (108.41,35.93) circle (0.30mm);
			\path[line width=0.15mm, draw=L, fill=F] (110.60,35.87) circle (0.30mm);
			\path[line width=0.15mm, draw=L, fill=F] (112.92,35.87) circle (0.30mm);
			\path[line width=0.15mm, draw=L, fill=F] (125.21,35.87) circle (0.50mm);
			\definecolor{L}{rgb}{0,0.392,0}
			\path[line width=0.15mm, draw=L, dash pattern=on 0.15mm off 0.50mm] (87.67,35.87) ellipse (6.48mm and 2.16mm);
			\path[line width=0.15mm, draw=L, dash pattern=on 0.15mm off 0.50mm] (101.04,36.12) ellipse (6.06mm and 2.24mm);
			\path[line width=0.15mm, draw=L, dash pattern=on 0.15mm off 0.50mm] (120.81,36.12) ellipse (6.73mm and 2.24mm);
			\definecolor{L}{rgb}{0,0,0}
			\path[line width=0.15mm, draw=L, fill=F] (118.38,35.93) circle (0.30mm);
			\path[line width=0.15mm, draw=L, fill=F] (120.56,35.87) circle (0.30mm);
			\path[line width=0.15mm, draw=L, fill=F] (122.89,35.87) circle (0.30mm);
			\definecolor{L}{rgb}{0,0,0.545}
			\path[line width=0.15mm, draw=L] (136.18,82.39) -- (142.66,59.13);
			\path[line width=0.15mm, draw=L] (145.15,82.22) -- (142.66,58.96);
			\path[line width=0.15mm, draw=L] (149.64,82.22) -- (147.31,58.96);
			\path[line width=0.15mm, draw=L] (158.61,82.22) -- (147.31,59.13);
			\path[line width=0.15mm, draw=L] (169.41,82.39) -- (165.75,59.13);
			\path[line width=0.15mm, draw=L] (178.38,82.55) -- (165.75,58.96);
			\definecolor{L}{rgb}{0,0.392,0}
			\path[line width=0.15mm, draw=L] (91.66,59.46) -- (83.18,36.20);
			\path[line width=0.15mm, draw=L] (91.49,59.13) -- (91.99,36.04);
			\path[line width=0.15mm, draw=L] (96.14,58.80) -- (96.14,36.20);
			\path[line width=0.15mm, draw=L] (96.14,59.13) -- (105.44,35.87);
			\path[line width=0.15mm, draw=L] (114.92,59.29) -- (116.24,36.20);
			\path[line width=0.15mm, draw=L] (114.92,59.29) -- (125.38,36.20);
			\path[line width=0.15mm, draw=L] (119.33,59.90) -- (131.95,81.49);
			\definecolor{F}{rgb}{0,0.392,0}
			\path[line width=0.15mm, draw=L, fill=F] (119.33,59.90) -- (120.34,60.93) -- (119.33,59.90) -- (119.73,61.28) -- (119.33,59.90) -- cycle;
			\path[line width=0.15mm, draw=L, fill=F] (131.95,81.49) -- (130.94,80.45) -- (131.95,81.49) -- (131.55,80.10) -- (131.95,81.49) -- cycle;
			\definecolor{L}{rgb}{0,0,0.545}
			\path[line width=0.15mm, draw=L] (138.26,59.07) -- (130.46,36.48);
			\definecolor{F}{rgb}{0,0,0.545}
			\path[line width=0.15mm, draw=L, fill=F] (138.26,59.07) -- (137.47,57.86) -- (138.26,59.07) -- (138.13,57.63) -- (138.26,59.07) -- cycle;
			\path[line width=0.15mm, draw=L, fill=F] (130.46,36.48) -- (131.25,37.69) -- (130.46,36.48) -- (130.58,37.92) -- (130.46,36.48) -- cycle;
			\draw(62.12,59.07) node[anchor=base west]{\fontsize{9}{11}\selectfont $2t$};
			\draw(150.09,48.96) node[anchor=base west]{\fontsize{9}{11}\selectfont $S_1:\overline{K_t}$};
			\draw(138.32,77.98) node[anchor=base west]{\fontsize{7}{8}\selectfont $t-1$};
			\draw(152.18,77.98) node[anchor=base west]{\fontsize{7}{8}\selectfont $t-1$};
			\draw(172.02,77.98) node[anchor=base west]{\fontsize{7}{8}\selectfont $t-1$};
			\draw(85.34,39.03) node[anchor=base west]{\fontsize{7}{8}\selectfont $t-1$};
			\draw(100.32,39.03) node[anchor=base west]{\fontsize{7}{8}\selectfont $t-1$};
			\draw(119.04,39.03) node[anchor=base west]{\fontsize{7}{8}\selectfont $t-1$};
			\definecolor{L}{rgb}{0.098,0.098,0.439}
			\path[line width=0.15mm, draw=L] (138.24,60.62) -- (112.78,81.18);
			\definecolor{F}{rgb}{0.098,0.098,0.439}
			\path[line width=0.15mm, draw=L, fill=F] (138.24,60.62) -- (137.37,61.78) -- (138.24,60.62) -- (136.93,61.23) -- (138.24,60.62) -- cycle;
			\path[line width=0.15mm, draw=L, fill=F] (112.78,81.18) -- (113.65,80.03) -- (112.78,81.18) -- (114.09,80.58) -- (112.78,81.18) -- cycle;
			\definecolor{L}{rgb}{0.333,0.42,0.184}
			\path[line width=0.15mm, draw=L] (119.00,58.74) -- (148.05,36.10);
			\definecolor{F}{rgb}{0.333,0.42,0.184}
			\path[line width=0.15mm, draw=L, fill=F] (119.00,58.74) -- (119.89,57.60) -- (119.00,58.74) -- (120.32,58.15) -- (119.00,58.74) -- cycle;
			\path[line width=0.15mm, draw=L, fill=F] (148.05,36.10) -- (147.16,37.24) -- (148.05,36.10) -- (146.73,36.69) -- (148.05,36.10) -- cycle;
			\definecolor{L}{rgb}{1,0,0}
			\path[line width=0.15mm, draw=L, dash pattern=on 0.15mm off 0.50mm] (81.90,53.25) rectangle +(95.39,12.28);
	\end{tikzpicture}}%
	\caption{The graph $G_{r,t}(0,t)$. The two copies of $F_{r,t}$ are joined only through the vertex cut $S_1\cup S_2$ of size $2t$; the labels indicate the prescribed local degrees.}
	\label{fig1}
\end{figure}
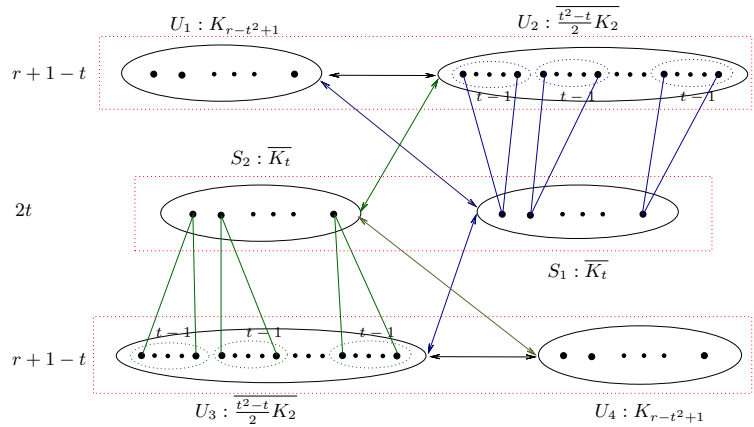

The graph $G_{r,t}(k_1,k_2)$ is not uniquely determined, since the required biregular bipartite graphs may have multiple realisations. Nevertheless, every such graph admits the equitable partition
\[
\{U_1,U_2,S_1,S_2,U_3,U_4\}.
\]
With rows and columns ordered as $(U_1,U_2,S_1,S_2,U_3,U_4)$, the quotient matrix of this partition is
{
\small
\begin{equation}\label{eq:Q6}
Q_{k_1,k_2}=
\setlength{\arraycolsep}{2pt}
\begin{pmatrix}
d-1 & t^{2}-t & t-k_{1} & k_{1} & 0 & 0\\
d & t^{2}-t-2 & t+1-k_{2} & k_{2} & 0 & 0\\
\frac{(t-k_{1})d}{t} & (t+1-k_{2})(t-1) & 0 & 0 & k_{2}(t-1) & \frac{k_{1}d}{t}\\
\frac{k_{1}d}{t} & k_{2}(t-1) & 0 & 0 & (t+1-k_{2})(t-1) & \frac{(t-k_{1})d}{t}\\
0 & 0 & k_{2} & t+1-k_{2} & t^{2}-t-2 & d\\
0 & 0 & k_{1} & t-k_{1} & t^{2}-t & d-1
\end{pmatrix}.
\end{equation}
}
The matrix $Q_{k_1,k_2}$ is invariant under reversal of the six parts. Consequently, the reversal-symmetric and reversal-antisymmetric subspaces are invariant under $Q_{k_1,k_2}$. Vectors in these two subspaces have the forms $(x,y,z,z,y,x)^{\top}$ and $(x,y,z,-z,-y,-x)^{\top}$, respectively.

For convenience, put 
\begin{equation}
\label{eq:atog}
a=d-1=r-t^2,\; b=t(t-1),\; e=t^2-t-2,\; f=\frac{d}{t},\; g=t-1
\end{equation}
and
$$
\alpha=t-2k_1,\; \beta=t+1-2k_2.
$$
\begin{samepage}
The restriction of $Q_{k_1,k_2}$ to the symmetric subspace is represented by
\begin{equation}
\label{eq:Msym}
M_{\mathrm{sym}}=
\begin{pmatrix}
a & b & t\\
d & e & t+1\\
d & t^2-1 & 0
\end{pmatrix}.
\end{equation}
\end{samepage}
The restriction to the antisymmetric subspace is represented by the matrix $M(\alpha,\beta)$ defined in~\eqref{eq:M}.
It follows that $\operatorname{spec}(Q_{k_1,k_2})=\operatorname{spec}(M_{\mathrm{sym}})\cup\operatorname{spec}(M(\alpha,\beta))$. A direct calculation gives $\det(xI-M_{\mathrm{sym}})=(x-r)(x+1)(x+t+1)$. Thus the eigenvalues of $M_{\mathrm{sym}}$ are $r$, $-1$, and $-(t+1)$, independently of $(k_1,k_2)$.

We next record two useful properties of the antisymmetric block. First, since $fb=\frac{d}{t}\,t(t-1)=d(t-1)=gd$, the positive diagonal matrix
\begin{equation}
\label{eq:H}
H=\operatorname{diag}(f,g,1)
\end{equation} 
satisfies $HM(\alpha,\beta)=M(\alpha,\beta)^{\top}H$. Consequently, $H^{1/2}M(\alpha,\beta)H^{-1/2}$ is real symmetric. Hence $M(\alpha,\beta)$ is diagonally similar to a real symmetric matrix and therefore has only real eigenvalues.

Second, let $\psi_{r,t}$ be the characteristic polynomial of the upper-left $2\times2$ principal submatrix of $M(\alpha,\beta)$. Then
\begin{eqnarray}
\psi_{r,t}(x) &= & x^2-(a+e)x+(ae-bd) \nonumber \\
&=& x^2-\bigl[(r-t^2)+(t^2-t-2)\bigr]x +(r-t^2)(t^2-t-2)-t(t-1)(r-t^2+1) \nonumber \\
&=& x^2-(r-t-2)x-(2r-t^2-t). \label{eq:psi}
\end{eqnarray}
In particular, $\det\bigl(xI-M(0,0)\bigr)=x\psi_{r,t}(x)$. The discriminant of $\psi_{r,t}$ is
\[
\begin{aligned}
\Delta_{\psi_{r,t}}
=(r-t-2)^2+4(2r-t^2-t)=(r+2)^2-t(3t+2r)=(r-t+2)^2-4t(t-1).
\end{aligned}
\]
Moreover, since $ae-bd=-(2r-t^2-t)<0$, the two roots of $\psi_{r,t}$ have opposite signs. It is easy to see that the larger root of $\psi_{r,t}$ is $\rho(r,t)$ defined in \eqref{eq:rho}.

\subsection{Proof of Theorem~\ref{thm:1.2}}

\begin{lem}\label{pro:4.1}
Let $t$ and $r$ be integers such that $t\geq2$, $r>2t^2$ and $t\mid(r+1)$. Then for all integers $k_1,k_2$ with $0\leq k_1\leq t$ and $1\leq k_2\leq t$, $G_{r,t}(k_1,k_2)$ is an $r$-regular graph with order $2(r+1)$ and vertex connectivity $\kappa\bigl(G_{r,t}(k_1,k_2)\bigr)=2t$.
\end{lem}

\begin{proof}
Write $G=G_{r,t}(k_1,k_2)$. By Construction~\ref{const:4.1}, $G$ has $2\bigl(d+t(t-1)\bigr)+2t=2(r+1)$ vertices. Every vertex in $U_1\cup U_4$ has degree $(r-t)+t=r$, every vertex in $U_2\cup U_3$ has degree $(r-t-1)+(t+1)=r$, and every vertex in $S_1\cup S_2$ has degree $d+(t-1)(t+1)=r$. Hence $G$ is $r$-regular.

Set $V_L=U_1\cup U_2$, $V_R=U_3\cup U_4$ and $S=S_1\cup S_2$. Recall from Construction~\ref{const:4.1} that $V_L$ and $V_R$ are the vertex sets of the two copies of $F_{r,t}$ and that there are no edges between them. Both $G[V_L]$ and $G[V_R]$ are connected. Moreover, by~\eqref{eq:degree-table} and $1\leq k_2\leq t$, every vertex of $S$ has at least $t-1>0$ neighbours in each of $V_L$ and $V_R$. Hence $G$ is connected. Since $|S|=2t$, deleting $S$ disconnects $G$ into the two nontrivial subgraphs $G[V_L]$ and $G[V_R]$. Thus, $\kappa(G)\leq2t$.

It remains to prove that no vertex cut of $G$ with size smaller than $2t$ exists. We first determine the vertex connectivity of the two copies of $F_{r,t}$. Set $m=t^2-t$. Then
$G[V_L]\cong G[V_R]\cong K_d\vee\overline{\tfrac m2K_2}$.
Let $Y$ be a set of fewer than $d+m-2$ vertices in one of these two graphs. If a vertex of the $K_d$-part remains, then it is adjacent to every other remaining vertex, and hence the remaining graph is connected. Suppose that the entire $K_d$-part is contained in $Y$. This case cannot occur when $m=2$. If $m\geq4$, then at least three vertices remain in the $\overline{\frac m2K_2}$-part. Any two nonadjacent remaining vertices are the end-vertices of one of the deleted matching edges and have every third remaining vertex as a common neighbour. Thus the remaining graph is again connected. On the other hand, deleting the $K_d$-part together with all but the two end-vertices of one deleted matching edge disconnects the graph. Therefore,
$\kappa(G[V_L])=\kappa(G[V_R])=d+m-2=r-t-1\geq2t$,
where the last inequality follows from $r>2t^2$ and $t\geq2$.

Suppose for a contradiction that there exists a vertex cut $X$ of $G$ such that $|X|<2t$. Then the two graphs $G[V_L\setminus X]$ and $G[V_R\setminus X]$ are nonempty and connected. For each $s\in S$, set $N_L(s)=N_G(s)\cap V_L$ and $N_R(s)=N_G(s)\cap V_R$. If some vertex of $S\setminus X$ has a neighbour in each of $V_L\setminus X$ and $V_R\setminus X$, then it joins the two connected graphs $G[V_L\setminus X]$ and $G[V_R\setminus X]$. Moreover, since $S$ is independent and every vertex of $S$ has degree $r>|X|$, every other vertex of $S\setminus X$ has a neighbour outside $X$. It would follow that $G-X$ is connected, contrary to the choice of $X$. Hence no vertex of $S\setminus X$ has a neighbour in both $V_L\setminus X$ and $V_R\setminus X$.

Define $B_L=\{s\in S\setminus X:N_L(s)\subseteq X\}$ and
$B_R=\{s\in S\setminus X:N_R(s)\subseteq X\}$. Since every vertex of $S\setminus X$ has a neighbour outside $X$ on exactly one of the two sides, $S\setminus X$ is the disjoint union of $B_L$ and $B_R$. Set $b_i=|B_L\cap S_i|$ for $i=1,2$. By Construction~\ref{const:4.1}, every vertex of $B_L\cap S_1$ has $(t-1)(t+1-k_2)$ neighbours in $U_2$, whereas every vertex of $U_2$ has $t+1-k_2$ neighbours in $S_1$. Since the entire left neighbourhood of each vertex in $B_L$ lies in $X$, counting the edges between $B_L\cap S_1$ and $U_2$ gives $|X\cap U_2|\geq(t-1)b_1$. Similarly, counting the edges between $B_L\cap S_2$ and $U_2$ gives $|X\cap U_2|\geq(t-1)b_2$. Consequently,
$|X\cap U_2|\geq(t-1)\max\{b_1,b_2\}$.

If $b_1b_2>0$, choose $s_i\in B_L\cap S_i$ for $i=1,2$. The sizes of their neighbourhoods in $U_1$ are $d(t-k_1)/t$ and $dk_1/t$, respectively, and both neighbourhoods are contained in $X$. Therefore, $|X\cap U_1|\geq\max\left\{\frac{d(t-k_1)}t,\frac{dk_1}t\right\}\geq\frac d2\geq t$.
Here $d/2\geq t$ because $t\mid d$ and $d>t+1$, and hence $d\geq2t$. It follows that
\[
|X\cap V_L|\geq(t-1)\max\{b_1,b_2\}
+t\mathbf{1}_{\{b_1b_2>0\}}
\geq b_1+b_2=|B_L|,
\]
where the second inequality follows from $0\leq b_1,b_2\leq t$. By the same argument on the right-hand side, $|X\cap V_R|\geq|B_R|$.

Set $q=|S\setminus X|$. Since $|X\cap S|=2t-q$ and $|B_L|+|B_R|=q$, we obtain $|X|\geq(2t-q)+|B_L|+|B_R|=2t$,
which contradicts $|X|<2t$. This contradiction shows that no vertex cut of $G$ with size smaller than $2t$ exists. Thus, $\kappa(G)\geq2t$, and therefore $\kappa(G)=2t$.
\end{proof}

Before determining the second largest eigenvalue of $G_{r,t}(k_1,k_2)$, we establish two estimates for $M(\alpha,\beta)$. The first gives the lower bound $\rho(r,t)$, and the second separates this bound from the eigenvalues arising outside the quotient subspace. We retain the notation $a,b,d,e,f,g$ and $\psi_{r,t}(x)=x^2-(a+e)x+(ae-bd)$ introduced in \eqref{eq:atog} and \eqref{eq:psi}, respectively. Recall that the larger root of $\psi_{r,t}$ is $\rho(r,t)$. Part (ii) of the following lemma implies that $\rho^*(r,t)\geq\rho(r,t)$, where $\rho^*(r,t)$ is defined in~\eqref{eq:rhostar}.

\begin{lem}\label{lem:4.2}
Let $t$ and $r$ be integers such that $t\geq2$ and $r>2t^2$, and, for real $\alpha$ and $\beta$, set $\Phi_x(\alpha,\beta)=f(x-e)\,\alpha^2+2dg\,\alpha\beta+g(x-a)\,\beta^2$.
\begin{itemize}
\item[{\rm(i)}]
$\det\bigl(xI-M(\alpha,\beta)\bigr)=x\,\psi_{r,t}(x)-\Phi_x(\alpha,\beta)$. Moreover, $\lambda_{\max}\bigl(M(\alpha,\beta)\bigr)\geq\rho(r,t)$ for all real $\alpha,\beta$, with equality at $(\alpha,\beta)=(0,0)$.

\item[{\rm(ii)}] $\rho(r,t)>\sqrt{r(t+1)}$.
\end{itemize}
\end{lem}

\begin{proof}
(i) A straightforward expansion gives
\begin{align*}
\det\bigl(xI-M(\alpha,\beta)\bigr)
&= x^3-(a{+}e)x^2+(ae{-}bd{-}\alpha^2 f{-}\beta^2 g)x
+ a\beta^2 g + \alpha^2 ef - 2\alpha\beta dg \\
&= x\psi_{r,t}(x)-\Phi_x(\alpha,\beta). \\
\end{align*}
For $(\alpha,\beta)=(0,0)$, the characteristic polynomial is $x\psi_{r,t}(x)$. Its largest root is $\rho_0:=\rho(r,t)$. Indeed, $ae-bd=-2r+t^2+t<0$, so $\psi_{r,t}$ has one positive and one negative root. For arbitrary real $\alpha$ and $\beta$, we have
$\det\bigl(\rho_0 I-M(\alpha,\beta)\bigr)=-\Phi_{\rho_0}(\alpha,\beta)$.
Since $\psi_{r,t}(a)=\psi_{r,t}(e)=-bd<0$, the larger root $\rho_0$ of $\psi_{r,t}$ exceeds both $a$ and $e$. Moreover, $\psi_{r,t}(\rho_0)=0$ implies
$(\rho_0-a)(\rho_0-e)=bd=t(t-1)d$. Therefore,
\begin{equation}\label{eq:Phi-rho-factor}
\Phi_{\rho_0}(\alpha,\beta)
=g(\rho_0-a)
\left(\beta+\frac{d}{\rho_0-a}\alpha\right)^2.
\end{equation}
Hence $\Phi_{\rho_0}(\alpha,\beta)\geq0$, and so
$\det(\rho_0 I-M(\alpha,\beta))\leq0$. On the other hand,
$\det(xI-M(\alpha,\beta))\to+\infty$ as $x\to+\infty$.
Since all eigenvalues of $M(\alpha,\beta)$ are real, it follows that
$\lambda_{\max}(M(\alpha,\beta))\geq\rho_0$. Equality is attained at
$(\alpha,\beta)=(0,0)$.

(ii) By~\eqref{eq:rho} and the identity
$(r+2)^2-t(3t+2r)=(r-t-2)^2+4\bigl(2r-t(t+1)\bigr)$, $\rho(r,t)$ is the larger root of $\psi_{r,t}(x)$. The smaller root of $\psi_{r,t}$ is negative. Therefore,
$\rho(r,t)>\sqrt{r(t+1)}$ if and only if $\psi_{r,t}(\sqrt{r(t+1)})<0$, which is equivalent to
$(r-t-2)\sqrt{r(t+1)}>r(t-1)+t(t+1)$. Both sides are positive because $r>t(t+1)$. After squaring, it remains to prove that $P(r)>0$, where $P(r):=(r-t-2)^2r(t+1)-\bigl(r(t-1)+t(t+1)\bigr)^2$.

The polynomial $P$ has leading coefficient $t+1>0$, and $\frac{dP}{dr}$ is an upward-opening quadratic whose vertex is $\frac{3t^2+4t+5}{3(t+1)}<t(t+1)+1$. A direct calculation gives $P\bigl(t(t+1)+1\bigr)=t(t^6+t^5-2t^4-6t^3-3t^2+t+4)>0$ and $\frac{dP}{dr}\bigl(t(t+1)+1\bigr)=3t^5+3t^4-4t^2+t-3>0$ for $t\geq2$. Upon writing $t=u+2$, both expressions become polynomials in $u$ with positive coefficients. Hence $P$ is increasing and positive on $[t(t+1)+1,\infty)$. Finally, since $r$ is an integer and $r>2t^2$, we have $r\geq2t^2+1\geq t(t+1)+1$. Thus $P(r)>0$, and the result follows.
\end{proof} 

\begin{proof}[Proof of the estimates in Remark~\ref{rem:rho-gap}(a)]
Fix $t$. For an admissible pair $(\alpha,\beta)$, denote the largest eigenvalue of $M(\alpha,\beta)$ by $\lambda_{\alpha,\beta}(r)$. Since the set of admissible pairs is finite, the error terms below are uniform over this set. Applying Gershgorin's theorem to the symmetric matrix $H^{1/2}M(\alpha,\beta)H^{-1/2}$ shows that $M(\alpha,\beta)$ has a unique eigenvalue asymptotic to $r$, while its other two eigenvalues are $O(\sqrt r)$, where $H$ is the matrix in \eqref{eq:H}. Hence this unique eigenvalue of $M(\alpha,\beta)$ is $\lambda_{\alpha,\beta}(r)$ for all sufficiently large $r$. Substituting $x=\rho(r,t)+s$ into the characteristic equation in Lemma~\ref{lem:4.2}(i) and comparing powers of $r^{-1}$, we obtain
\[
\lambda_{\alpha,\beta}(r)-\rho(r,t)
=\frac{\alpha^2}{t}
+\frac{2(t-1)\alpha\beta
-\frac{\alpha^2}{t}(2t^2-2t-1)-\frac{\alpha^4}{t^2}}{r}
+O(r^{-2}).
\]
If $\alpha=0$, expanding one term further gives
\[
\lambda_{0,\beta}(r)-\rho(r,t)
=\frac{t(t-1)^2\beta^2}{r^2}+O(r^{-3}).
\]
If $t$ is even, then $\alpha$ is even and $\beta$ is odd. The minimum is therefore attained, for all sufficiently large $r$, when $\alpha=0$ and $|\beta|=1$, which gives
\[
\rho^*(r,t)-\rho(r,t)
=\frac{t(t-1)^2}{r^2}+O(r^{-3}).
\]
If $t$ is odd, then $\alpha$ is odd and $\beta$ is even. The leading term first forces $|\alpha|=1$; among these pairs, the coefficient of $r^{-1}$ is minimised when $\alpha\beta=-(t-1)$. Consequently,
\[
\rho^*(r,t)-\rho(r,t)=\frac1t+O(r^{-1}).
\]
This proves the estimates in Remark~\ref{rem:rho-gap}(a).
\end{proof}

We next determine the second largest eigenvalue of the graphs in Construction~\ref{const:4.1}.

\begin{lem}\label{lem:4.3}
Let $t$ and $r$ be integers such that $t\geq2$, $r>2t^2$ and $t\mid(r+1)$, and let $k_1$ and $k_2$ be integers such that $0\leq k_1\leq t$ and $1\leq k_2\leq t$. Then
$\lambda_2\bigl(G_{r,t}(k_1,k_2)\bigr)=\lambda_{\max}\!\bigl(M(t-2k_1,t+1-2k_2)\bigr)$, where $M(\alpha,\beta)$ is defined in~\eqref{eq:M}. In particular, $\lambda_2(G^*_{r,t})=\rho^*(r,t)$.
\end{lem}

\begin{proof}
Write $G=G_{r,t}(k_1,k_2)$, and set $\alpha=t-2k_1$ and $\beta=t+1-2k_2$. By Construction~\ref{const:4.1}, $\mathcal P=\{U_1,U_2,S_1,S_2,U_3,U_4\}$ is an equitable partition of $G$ with quotient matrix $Q_{k_1,k_2}$ given by~\eqref{eq:Q6}. Thus every eigenvalue of $Q_{k_1,k_2}$ is an eigenvalue of $G$ by Theorem~\ref{thm:2.2}.

As noted above, the reversal symmetry of $Q_{k_1,k_2}$ decomposes its spectrum into the spectra of a symmetric block and an antisymmetric block. The symmetric block is the matrix $M_{\mathrm{sym}}$ defined in \eqref{eq:Msym} and $\det(xI-M_{\mathrm{sym}})=(x-r)(x+1)(x+t+1)$. Hence $\operatorname{spec}(M_{\mathrm{sym}})=\{r,-1,-(t+1)\}$. The corresponding eigenvectors lift to vectors of the form $(x,y,z,z,y,x)$; in particular, the eigenvector for $r$ is $\mathbf{1}$. The antisymmetric block is $M(\alpha,\beta)$, and its eigenvectors lift to vectors of the form $(x,y,z,-z,-y,-x)$. Since the paired parts have equal orders, every such lifted vector has coordinate sum zero and is therefore orthogonal to $\mathbf{1}$.

It remains to bound the eigenvalues of $G$ that do not arise from the quotient. Let $\mathcal S\subseteq\mathbb R^{V(G)}$ be the six-dimensional space of vectors that are constant on each part of $\mathcal P$. Since $\mathcal P$ is equitable, $\mathcal S$ is invariant under $A=A(G)$. The matrix $A$ is symmetric, so $\mathcal S^{\perp}$ is also invariant. Therefore, the remaining eigenvalues of $G$ are precisely the eigenvalues of $A|_{\mathcal S^{\perp}}$.

Let $w\in\mathcal S^{\perp}$ be a unit vector. The sum of the coordinates of $w$ on each part of $\mathcal P$ is zero. We estimate $w^{\top}Aw$ by considering the different types of edges. The contributions from edges within the parts are non-positive. Indeed, since $G[U_1]\cong K_d$ (and similarly $G[U_4]\cong K_d$), we have $w|_{U_1}^{\top}(J-I)w|_{U_1}=-\|w|_{U_1}\|^{2}$. The adjacency matrix of $G[U_2]\cong\overline{\frac{t^2-t}{2}K_2}$ is $J-I-P$, where $P$ is the adjacency matrix of the deleted perfect matching. The same holds for $U_3$. Since $w|_{U_2}$ has coordinate sum zero and $\|P\|_{\mathrm{op}}=1$,
\[
w|_{U_2}^{\top}(J-I-P)\,w|_{U_2}
=-\|w|_{U_2}\|^{2}-w|_{U_2}^{\top}P\,w|_{U_2}\leq 0.
\]
The independent sets $S_1$ and $S_2$ contribute zero. The complete bipartite edge sets between $U_1$ and $U_2$, and between $U_3$ and $U_4$, also contribute zero, because $w|_{U_i}^{\top}Jw|_{U_j}=(\mathbf{1}^{\top}w|_{U_i})(\mathbf{1}^{\top}w|_{U_j})=0$ for $(i,j)=(1,2)$ and $(3,4)$.

Let $A_{US}$ be the biadjacency matrix between $U=U_1\cup\cdots\cup U_4$ and $S=S_1\cup S_2$. Every column of $A_{US}$ sums to $r$, because $S$ is independent and $G$ is $r$-regular, while every row sums to at most $t+1$. The standard inequality $\|A_{US}\|_{\mathrm{op}}\leq\sqrt{\|A_{US}\|_1\|A_{US}\|_\infty}$ gives $\|A_{US}\|_{\mathrm{op}}\leq\sqrt{r(t+1)}$. Using $2\|w_U\|\|w_S\|\leq\|w_U\|^{2}+\|w_S\|^{2}=1$, we obtain
\[
w^{\top}Aw\leq2w_U^{\top}A_{US}w_S
\leq2\|A_{US}\|_{\mathrm{op}}\|w_U\|\|w_S\|
\leq\sqrt{r(t+1)}.
\]
It follows from the Rayleigh quotient that every eigenvalue of $A|_{\mathcal S^{\perp}}$ is at most $\sqrt{r(t+1)}$. On the other hand, Lemma~\ref{lem:4.2} gives $\lambda_{\max}(M(\alpha,\beta))\geq\rho(r,t)>\sqrt{r(t+1)}$. Hence every eigenvalue outside the quotient subspace is strictly smaller than $\lambda_{\max}(M(\alpha,\beta))$.

The other quotient eigenvalues are $-1$, $-(t+1)$ and the two smaller eigenvalues of $M(\alpha,\beta)$, each of which is at most $\lambda_{\max}(M(\alpha,\beta))$. Moreover, an eigenvector of $M(\alpha,\beta)$ corresponding to $\lambda_{\max}(M(\alpha,\beta))$ lifts to an eigenvector of $G$ orthogonal to $\mathbf{1}$. Since $G$ is connected and $r$-regular, the eigenvalue $r$ is simple and has eigenvector $\mathbf{1}$. Thus $\lambda_{\max}(M(\alpha,\beta))<r$, and we conclude that $\lambda_2(G_{r,t}(k_1,k_2))=\lambda_{\max}(M(t-2k_1,t+1-2k_2))$. Finally, taking $(k_1,k_2)=(k_1^*,k_2^*)$ and applying~\eqref{eq:rhostar} gives $\lambda_2(G^*_{r,t})=\rho^*(r,t)$.
\end{proof}

\begin{proof}[Proof of Theorem~\ref{thm:1.2}]
Let $t \ge 2$ and $r > 2t^2$ be integers such that $t\mid(r+1)$. Then Construction~\ref{const:4.1} applies. By Lemmas \ref{pro:4.1} and \ref{lem:4.3}, the graph $G^*_{r,t}$ is connected and $r$-regular, and it satisfies $\kappa(G^*_{r,t})=2t$ and $\lambda_2(G^*_{r,t})=\rho^*(r,t)$. Thus $G^*_{r,t}$ is admissible in the infimum defining $\lambda_2^{\min}(r,t)$ in~\eqref{eq:lmin}, and hence $\lambda_2^{\min}(r,t)\leq\rho^*(r,t)$. 
\end{proof}

We conclude this section with two observations. First, the divisibility condition $t\mid(r+1)$ guarantees that all prescribed degrees are integral. Second, although $(\alpha,\beta)=(0,0)$ minimises the real-parameter problem in Lemma~\ref{lem:4.2}, it is not an admissible integer pair: $\alpha=t-2k_1$ has the same parity as $t$, whereas $\beta=t+1-2k_2$ has the opposite parity. Thus the difference $\rho^*(r,t)-\rho(r,t)$ arises from restricting the real parameter space to the integer pairs allowed by the construction. Theorem~\ref{thm:1.2} provides an upper bound obtained from the constructed family. It does not claim that $\lambda_2^{\min}(r,t)=\rho^*(r,t)$, nor does it imply uniqueness of an extremal graph if equality holds.

\section{Proof of Theorem~\ref{thm:1.3}}
\label{sec:pf2}

Let $t\geq 2$ and $r>(2t-1)(2t-2)$ be integers, and let $G$ be a connected $r$-regular graph with $\kappa(G)=2l-1\leq2t-1$. We aim to prove the lower bound in~\eqref{eq:odd-bound}.

In the case $l=1$, by Fiedler's inequality, $\lambda_2(G)\geq r-1$. On the other hand, $r-1-\xi(r,t)=\frac{2rt-3r+2t-4}{2(r-t+2)}>0$ as $t\geq2$. Since $\min\bigl\{\xi(r,t),\nu(r,t)\bigr\}\leq \xi(r,t)<r-1$, the result follows when $l = 1$.

For the remainder of this section, we assume $2\leq l\leq t$ and put $s=2l-1$. To prove Theorem \ref{thm:1.3}, we first minimise the largest nontrivial eigenvalue of a three-part quotient, then show that the resulting bound decreases with the cut parameter $l$, and finally compare the two candidate minima.

\subsection{Minimising the three-part quotient}
\label{subsec:odd-lower}

Let $S$ be a minimum vertex cut of $G$, so $|S|=s=2l-1$. Let $C_1,C_2,\ldots,C_k$ be the components of $G-S$, where $k\geq2$. Set $V_1=C_1$ and $V_2=\bigcup_{i=2}^{k}C_i$. Denote $m_i=|V_i|$ for $i=1,2$, and set $p=|[S,V_1]|$ and $q=|[S,V_2]|$. Then the quotient matrix of the partition $\{V_1,S,V_2\}$ of $V(G)$ is
\[
Q_0=
\begin{pmatrix}
r-\frac{p}{m_1} & \frac{p}{m_1} & 0 \\[6pt]
\frac{p}{s} & r-\frac{p+q}{s} & \frac{q}{s} \\[6pt]
0 & \frac{q}{m_2} & r-\frac{q}{m_2}
\end{pmatrix}.
\]
Applying Theorem~\ref{thm:2.1} to remove the trivial eigenvalue $r$ yields the matrix
\begin{equation}\label{eq:odd-Q1}
Q_1=
\begin{pmatrix}
r-\frac{p}{m_1}-\frac{p}{s} & \frac{q}{s} \\[6pt]
\frac{p}{s} & r-\frac{q}{m_2}-\frac{q}{s}
\end{pmatrix}
\end{equation}
and $\lambda_1(Q_1)=\lambda_2(Q_0)$. Theorem~\ref{thm:2.2} now gives
\begin{equation}\label{eq:odd-interlace}
\lambda_2(G)\geq\lambda_2(Q_0)=\lambda_1(Q_1).
\end{equation}
By Lemma~\ref{lem:2.3}(ii), the sizes of $V_1$ and $V_2$ satisfy
\begin{equation}\label{eq:odd-msizes}
m_1,\ m_2\geq r-2l+2=r-s+1.
\end{equation}
Counting the edges from $S$ to $V_1$ and $V_2$ gives
\begin{equation}\label{eq:odd-admissible}
p\geq g(m_1),\quad q\geq g(m_2),\quad p+q\leq sr,
\quad\text{where } g(x):=x(r-x+1).
\end{equation}

For convenience, write~\eqref{eq:odd-Q1} as
\[
Q_1=\begin{pmatrix} a_1 & q/s\\ p/s & a_2\end{pmatrix},\quad
a_1=r-\frac{p}{m_1}-\frac{p}{s},\quad a_2=r-\frac{q}{m_2}-\frac{q}{s}.
\]
Since the off-diagonal entries are nonnegative, the larger eigenvalue of $Q_1$ is
\begin{equation}\label{eq:odd-eig}
\lambda_1(Q_1)=\frac{a_1+a_2}{2}
+\sqrt{\Bigl(\frac{a_1-a_2}{2}\Bigr)^2+\frac{pq}{s^2}}\,.
\end{equation}

\begin{lem}\label{lem:5.1}
For fixed nonnegative real numbers $p$ and $q$, the value $\lambda_1(Q_1)$ is non-decreasing in each of the positive variables $m_1$ and $m_2$.
\end{lem}

\begin{proof}
If $p=q=0$, then $Q_1=rI_2$, and the assertion is immediate. Suppose that $p+q>0$, and set $R=\sqrt{((a_1-a_2)/2)^2+pq/s^2}$. Then $R>0$: indeed, if $p=0<q$, then $a_1=r>a_2$, and the case $q=0<p$ is symmetric. By~\eqref{eq:odd-eig},
$\frac{\partial\lambda_1(Q_1)}{\partial a_1}=\frac12+\frac{a_1-a_2}{4R}\geq0$, because $R\geq\tfrac12|a_1-a_2|$. Moreover, $\partial a_1/\partial m_1=p/m_1^2\geq0$, while $a_2$ is independent of $m_1$. Hence the chain rule gives $\partial\lambda_1(Q_1)/\partial m_1\geq0$. By symmetry, $\partial\lambda_1(Q_1)/\partial m_2\geq0$ as well.
\end{proof}

To obtain a lower bound on $\lambda_1(Q_1)$, we now enlarge the feasible set by keeping $m_1$ and $m_2$ integral but allowing $p$ and $q$ to be nonnegative real numbers, subject to the constraints in~\eqref{eq:odd-admissible}.
For the parameters arising from the connected graph $G$, both $p$ and $q$ are positive.
In this relaxation, the derivative calculation in Claim~1 of the proof of Theorem~\ref{thm:1.1}, with $2l$ replaced by $s=2l-1$, shows that $\lambda_1(Q_1)$ is strictly decreasing in both $p$ and $q$ whenever $p,q>0$.
Thus, if $p+q<sr$, we may increase one or both parameters until $p+q=sr$ without increasing $\lambda_1(Q_1)$.

We may further assume that $m_1,m_2\leq r+1$. Indeed, $g(r+1)=0$, so replacing any $m_i>r+1$ by $r+1$ preserves the relaxed feasibility constraints. By Lemma~\ref{lem:5.1}, this replacement does not increase $\lambda_1(Q_1)$. Therefore, it suffices to minimise $\lambda_1(Q_1)$ over the region
\begin{equation}\label{eq:odd-compact}
m_1, m_2\in\mathbb{Z},\;\, r-2l+2 \leq m_1, m_2 \leq r+1,\;\,
g(m_1)\leq p,\;\, g(m_2)\leq q,\;\, p+q=sr.
\end{equation}
For each fixed integer pair $(m_1,m_2)$, the constraints in~\eqref{eq:odd-compact} define a possibly empty compact interval for $p$, with $q=sr-p$. The entire region is nonempty and contains only finitely many choices of $(m_1,m_2)$. Hence the minimum of $\lambda_1(Q_1)$ over the region~\eqref{eq:odd-compact} exists.

\begin{lem}\label{lem:5.2}
Let $l\geq2$ and $r>(2l-1)(2l-2)$ be integers. For real $a\in[r-2l+2,r-l+1]$, set $m_1=a$, $m_2=2r-2l+3-a$, $p=g(a)$ and $q=sr-g(a)$ in $Q_1$, and denote the resulting value of $\lambda_1(Q_1)$ by $\phi(a)$. Then for integral $a$, $m_2$ is the smallest feasible integer size of $V_2$. Moreover, $\phi$ is strictly decreasing on $[r-2l+2,r-l+1]$, and hence its minimum is attained at $a=r-l+1$.
\end{lem}

\begin{proof}
Put $b=2r-2l+3-a$. Then $a+b=2r-2l+3$ and $r-l+2\leq b\leq r+1$. We first show that, when $a$ is integral, $b$ is the smallest feasible size of $V_2$.

On the interval $[r-2l+2,r+1]$, we have $\frac{dg}{dx}=r+1-2x\leq-r+4l-3<0$, where the last inequality follows from $r>(2l-1)(2l-2)$. Hence $g$ is strictly decreasing on this interval. The expression $sr-g(a)-g(b)$ is a quadratic polynomial in $a$ with positive leading coefficient and discriminant $-4(2l-1)(2l-3)<0$. Therefore, $sr-g(a)-g(b)>0$, and so $g(b)<q$. Thus $b$ is feasible.

Similarly, $sr-g(a)-g(b-1)$ is a quadratic polynomial in $a$ with positive leading coefficient. At the endpoints $a=r-2l+2$ and $a=r-l+1$ of the interval $[r-2l+2, r-l+1]$, its values are $(2l-1)(2l-2)-r$ and $2l(l-1)-r$, respectively. Both are negative. By convexity, this quadratic is negative throughout the interval, and hence $g(b-1)>q$. Since $g$ is strictly decreasing, $b$ is the smallest integer satisfying $g(b)\leq q$. Lemma~\ref{lem:5.1} then shows that this choice minimises $\lambda_1(Q_1)$ over all feasible sizes of $V_2$.

It remains to show that $\phi$ is strictly decreasing. Regard $\tau=\operatorname{tr}Q_1$ and $D=\det Q_1$ as functions of the real variable $a$. Since $\phi=\tfrac12(\tau+\sqrt{\tau^2-4D})$ and $\sqrt{\tau^2-4D}=2\phi-\tau>0$, differentiation gives
\begin{equation}\label{eq:phiprime}
\frac{d\phi}{da}=\frac{\frac{d\tau}{da}\,\phi-\frac{dD}{da}}{\sqrt{\tau^2-4D}},
\qquad\text{so}\qquad
\operatorname{sign}\frac{d\phi}{da}=\operatorname{sign}\left(\frac{d\tau}{da}\,\phi-\frac{dD}{da}\right).
\end{equation}
The inequalities proved above, together with $g(b-1)<rb$, give $0<q<rb$. Also,  $p=g(a)>0$. Hence every entry of $Q_0$ is nonnegative, and its positive off-diagonal entries make it irreducible. Since $Q_0$ has constant row sum $r$, the Perron--Frobenius theorem gives $\phi=\lambda_1(Q_1)=\lambda_2(Q_0)<r$. Write $a=r-l+1-j$, where $0\leq j\leq l-1$. Differentiating the trace and determinant of~\eqref{eq:odd-Q1} and collecting over common positive denominators gives
\begin{align*}
\frac{d\tau}{da}
&=\frac{2\,\bigl(2jr+r+j^2-2jl+4j-l^2+l+1\bigr)}{(r-l+j+2)^2},\\
\frac{dD}{da}&-r\,\frac{d\tau}{da}
=\frac{2(r+1)\,N(j,l,r)}{(2l-1)\,(r-l+j+2)^2},
\end{align*}
where
\[
N(j,l,r)=(2j+1)r^2-\bigl(2j^2+8jl-6j+2l^2-2\bigr)r
+\bigl(4l^3-9l^2+4l+4j+6jl^2-14jl-5j^2-2j^3\bigr).
\]
The numerator of $\frac{d\tau}{da}$ is increasing in $j$, because its derivative with respect to $j$ is $2(r-l+j+2)>0$. At $j=0$, this numerator is
$r-l^2+l+1\geq(2l-1)(2l-2)-l^2+l+1=3l^2-5l+3>0$. Thus $\frac{d\tau}{da}>0$. Since $\phi<r$, it follows that
\[
\frac{d\tau}{da}\,\phi-\frac{dD}{da}
\;<\;\frac{d\tau}{da}\,r-\frac{dD}{da}
\;=\;-\,\frac{2(r+1)\,N(j,l,r)}{(2l-1)\,(r-l+j+2)^2}.
\]
It remains to show that $N(j,l,r)>0$. Put $R_0=(2l-1)(2l-2)$. As a quadratic polynomial in $r$, $N(j,l,r)$ has positive leading coefficient $2j+1$ and vertex
\[
r_{\mathrm v}=\frac{2j^2+8jl-6j+2l^2-2}{2(2j+1)}.
\]
Putting $R_0-r_{\mathrm v}$ over a common denominator and simplifying gives
\[
R_0-r_{\mathrm v}=\frac{3(l-1)^2+j(8l^2-16l+7)-j^2}{2j+1}.
\]

The numerator in the last expression is a concave quadratic polynomial in $j$. Its values at $j=0$ and $j=l-1$ are $3(l-1)^2$ and $(l-1)(8l^2-14l+5)$, respectively, and both are nonnegative. Hence it is nonnegative throughout $[0,l-1]$, so $r_{\mathrm v}\leq R_0$. Since $r$ is an integer and $r>R_0$, we have $r\geq R_0+1$. Therefore, $N(j,l,r)$ is increasing as a function of $r$ on the relevant interval, and $N(j,l,r)\geq N(j,l,R_0+1)$.
It remains to show that $N(j,l,R_0+1)>0$ for $0\leq j\leq l-1$. Differentiating with respect to $j$ gives
\[
\frac{d}{dj}N(j,l,R_0+1)
=-6j^2-(16l^2-24l+22)\,j
+32l^4-128l^3+198l^2-146l+40.
\]
This derivative is strictly decreasing for $j\geq0$. It is positive at $j=0$, because the constant $32l^4-128l^3+198l^2-146l+40$ is increasing for $l\geq2$ and equals $28$ at $l=2$. Consequently, $N(j,l,R_0+1)$ first increases and then, if its derivative changes sign, decreases on $[0,l-1]$. Its minimum is therefore attained at an endpoint. The two endpoint values are $(l-1)(8l^3-24l^2+29l-15)$ and $2(l-1)(2l-1)(8l^3-28l^2+34l-17)$. Both are positive for $l\geq2$, since the two cubic factors are increasing in $l$ and take the values $11$ and $3$, respectively, at $l=2$. Thus $N(j,l,r)>0$. By~\eqref{eq:phiprime}, we have $\frac{d\phi}{da}<0$ on $[r-2l+2, r-l+1]$, and hence $\phi$ is strictly decreasing on this interval. This completes the proof.
\end{proof}

For integral $a$, we refer to the parameters in Lemma \ref{lem:5.2} as a one-sided clique configuration, since in any graph realising them, $V_1$ induces a clique. 

\begin{lem}
\label{lem:cl1}
Let $l \ge 2$ and $r>(2l-1)(2l-2)$ be integers. Then on the region~\eqref{eq:odd-compact},
\[
\lambda_1(Q_1)\geq\min\bigl\{\xi(r,l),\nu(r,l)\bigr\},
\]
where $\xi(r,l)$ and $\nu(r,l)$ are defined in~\eqref{eq:xi-nu-def}. Moreover, the minimum over this relaxed region is attained at one of the following two configurations:
\begin{enumerate}[\rm(i)]
\item \emph{(balanced partition)} $m_1=m_2=r-l+2$ and $p=q=\tfrac{sr}{2}$, where $\lambda_1(Q_1)=\xi(r,l)$; and
\item \emph{(one-sided clique configuration)} up to interchanging the indices $1$ and $2$, $m_1=r-l+1$, $m_2=r-l+2$, $p=l(r-l+1)$ and $q=(l-1)(r+l)$, where $\lambda_1(Q_1)=\nu(r,l)$.
\end{enumerate}
\end{lem}

\begin{proof}
By the compactness argument above, $\lambda_1(Q_1)$ attains its minimum over the relaxed region~\eqref{eq:odd-compact}. Since $q=sr-p$ throughout this region, we first study the dependence on $p$ for fixed $m_1$ and $m_2$.

Put $q=sr-p$ and $\alpha_i=\frac1{m_i}+\frac1s$. Then both $a_1-a_2=sr\,\alpha_2-(\alpha_1+\alpha_2)p$ and $\tfrac12(a_1+a_2)$ are affine functions of $p$. The expression $\Delta(p):=(a_1-a_2)^2+\frac{4pq}{s^2}$ is a quadratic polynomial in $p$ with leading coefficient
$(\alpha_1+\alpha_2)^2-\tfrac4{s^2}=\bigl(\tfrac1{m_1}+\tfrac1{m_2}\bigr)\bigl(\tfrac1{m_1}+\tfrac1{m_2}+\tfrac4s\bigr)>0$. Moreover, $\operatorname{disc}_p\Delta=-\frac{16r^2}{s^2}\bigl(s^2\alpha_1\alpha_2-1\bigr)<0$, because $s^2\alpha_1\alpha_2=\left(1+\frac{s}{m_1}\right)\left(1+\frac{s}{m_2}\right)>1$.
Thus $\Delta(p)>0$ and $\frac{d^2}{dp^2}\sqrt{\Delta(p)}=-\frac{\operatorname{disc}_p\Delta}{4\Delta(p)^{3/2}}>0$. It follows from~\eqref{eq:odd-eig} that $\lambda_1(Q_1)$ is strictly convex in $p$. Consequently, on every feasible interval, the minimum is attained either at an endpoint or at the unique stationary point, if one exists. We now consider two cases.

\textbf{Case 1.} $m_1=m_2=m$. 

In this case, $Q_1$ is invariant under $(p,q)\mapsto(q,p)$. Hence the convex function $\lambda_1(Q_1)$ is symmetric about $p=sr/2$ and attains its minimum there. For $p=q=sr/2$, the diagonal entries of $Q_1$ are $r-\tfrac{sr}{2m}-\tfrac r2$, and the off-diagonal entries are $r/2$. Therefore, $\lambda_1(Q_1)=r-\tfrac{sr}{2m}$. This value is increasing in $m$, so it remains to determine the smallest feasible common size. The point $p=sr/2$ is feasible if and only if $g(m)\leq sr/2$. Now $g(r-l+1)-\tfrac{sr}{2}=\tfrac r2-l(l-1)>0$, while $\tfrac{sr}{2}-g(r-l+2)=\tfrac r2+(l-1)(l-2)>0$; both inequalities follow from $r>(2l-1)(2l-2)$. Hence the smallest feasible common size is $m=r-l+2$, and the corresponding value is $\lambda_1(Q_1)=r-\tfrac{sr}{2(r-l+2)}=\xi(r,l)$ by~\eqref{eq:xi-nu-def}.

\textbf{Case 2.} $m_1\ne m_2$. 

Suppose first that $m_1,m_2\geq r-l+2$. By Lemma~\ref{lem:5.1}, replacing both $m_1$ and $m_2$ by $r-l+2$, while keeping $p$ and $q$ fixed, does not increase $\lambda_1(Q_1)$. The new parameter tuple is used only for comparison and need not be feasible. For $m_1= m_2$, the convexity and symmetry argument in Case~1 shows that the value at any $p\in[0,sr]$ is at least its value at $p=q=sr/2$, namely $\xi(r,l)$. Thus $\lambda_1(Q_1)\geq \xi(r,l)$.

It remains to consider the case in which the smaller part, say $V_1$, has $m_1\leq r-l+1$. Then $g(m_1)\geq g(r-l+1)=l(r-l+1)>sr/2$. Since $\lambda_1(Q_1)$ is strictly convex in $p$, its derivative is strictly increasing. We claim that this derivative is already positive at $p=sr/2$. Let $R$ denote the square-root term in~\eqref{eq:odd-eig}. At $p=q=sr/2$, differentiation gives
\[
\frac{\partial\lambda_1(Q_1)}{\partial p}\Big|_{p=sr/2}
=\frac{\alpha_1-\alpha_2}{2}\left(\frac{(\alpha_1+\alpha_2)sr}{4R}-1\right).
\]
Since $m_1<m_2$, we have $\alpha_1>\alpha_2$. The expression in parentheses is positive if $(\alpha_1+\alpha_2)sr>4R$. Using $4R=r\sqrt{s^2(\alpha_1-\alpha_2)^2+4}$ and squaring both sides, this inequality is equivalent to $s^2\alpha_1\alpha_2>1$. The latter holds because $s^2\alpha_1\alpha_2=(1+\tfrac{s}{m_1})(1+\tfrac{s}{m_2})>1$. Therefore, $\partial\lambda_1(Q_1)/\partial p>0$ at $p=sr/2$.

Every feasible value of $p$ satisfies $p\geq g(m_1)>sr/2$. Since the derivative is strictly increasing, $\lambda_1(Q_1)$ is strictly increasing throughout the feasible interval $[g(m_1),sr-g(m_2)]$. Hence the minimum is attained at $p=g(m_1)$. By Lemma~\ref{lem:5.2}, the smallest feasible size of $V_2$ is $2r-2l+3-m_1$, and $\phi(m_1)$ is strictly decreasing. Thus the minimum is attained at $m_1=r-l+1$, which gives configuration~\textup{(ii)}.

In configuration~\textup{(ii)}, $p=g(r-l+1)=l(r-l+1)$, and hence $q=sr-p=(l-1)(r+l)$. Substituting these values, together with $m_1=r-l+1$, $m_2=r-l+2$ and $s=2l-1$, into~\eqref{eq:odd-Q1}, we obtain the diagonal entries of $Q_1$:
\[
a_1=r-l-\frac{l(r-l+1)}{2l-1},\quad
a_2=r-\frac{(l-1)(r+l)}{r-l+2}-\frac{(l-1)(r+l)}{2l-1}.
\]
The off-diagonal entries of $Q_1$ are $q/s$ and $p/s$. Since $l(r-l+1)+(l-1)(r+l)=(2l-1)r$, the trace $a_1+a_2$ simplifies to $\tau(r,l)$ as defined in \eqref{eq:tau}, and the discriminant is
\begin{equation}\label{eq:odd-discriminant}
\tau^2-4\det Q_1
 =\frac{(r+l)(r+1)\,W(r,l)}{(2l-1)(r-l+2)^2}.
\end{equation}
It follows from~\eqref{eq:xi-nu-def} and~\eqref{eq:odd-discriminant} that $\lambda_1(Q_1)=\nu(r,l)$ in configuration~\textup{(ii)}. These two cases exhaust all possible minima, and the result follows.
\end{proof}

\subsection{Proof of Theorem~\ref{thm:1.3}}
\label{subsec:odd-closed}

\begin{lem}\label{lem:5.3}
Fix integers $t\geq2$ and $r>(2t-1)(2t-2)$. Then both $\xi(r,l)$ and $\nu(r,l)$ are strictly decreasing in $l$ for real $l\in[2,t]$. Hence so is $\min\bigl\{\xi(r,l),\nu(r,l)\bigr\}$, and therefore
\begin{equation}\label{eq:odd-collapse}
	\min_{2\leq l\leq t}\min\bigl\{\xi(r,l),\nu(r,l)\bigr\}=\min\bigl\{\xi(r,t),\nu(r,t)\bigr\}.
\end{equation}
\end{lem}

\begin{proof}
By differentiating~\eqref{eq:xi-nu-def}, we obtain $\frac{d\xi}{dl}=-\frac{r(2r+3)}{2(r-l+2)^2}<0$ and hence $\xi(r,l)$ is strictly decreasing in $l$.  
	
For configuration~\textup{(ii)} in Lemma \ref{lem:cl1}, let $D(r,l)=\det Q_1$, and extend $\tau(r,l)$, $D(r,l)$ and $\nu(r,l)$ to real $l\in[2,t]$. Since $r>(2t-1)(2t-2)$, we have $r>2l(l-1)+1$ for every $l\in[2,t]$.

We first show that $\tau^2-4D>0$. The positivity statement in~\eqref{eq:W-positive} applies for every $l\in[2,t]$ and gives $W(r,l)>0$. Since the remaining factor in~\eqref{eq:odd-discriminant} is positive, we have $\tau^2-4D>0$. Equations~\eqref{eq:xi-nu-def} and~\eqref{eq:odd-discriminant} show that $\nu$ is the larger root of $z^2-\tau z+D=0$. Implicit differentiation of $\nu^2-\tau\nu+D=0$ gives
$$
\frac{d\nu}{dl}=\frac{\frac{d\tau}{dl}\nu-\frac{dD}{dl}}{2\nu-\tau}.
$$
Since $\frac{d\tau}{dl}=-2(r+1)^2/(r-l+2)^2<0$ and $\nu>\tau/2$, we have
$\frac{d\tau}{dl}\,\nu-\frac{dD}{dl}
<\frac12\,\frac{d\tau}{dl}\,\tau-\frac{dD}{dl} :=-C(r,l)$.
Substituting the explicit formulas for $\tau$ and $D$, differentiating, and collecting over the common positive denominator gives $C(r,l)=\frac{\Pi(r,l)}{(2l-1)^2(r-l+2)^3}$, where, for $y=r-2l(l-1)-1>0$,
\begin{align*}
\Pi(r,l)=\;&y^4+(12l^2-15l+11)y^3
+2(18l^4-41l^3+63l^2-48l+20)y^2\\
&+2(16l^6-48l^5+114l^4-161l^3+159l^2-93l+29)y\\
&+4(l^2-l+1)(4l^5-11l^3+23l^2-19l+7).
\end{align*}
Substituting $l=u+1$, where $u\geq1$, we obtain
\begin{align*}
\Pi(r,l)=\;&y^4+(12u^2+9u+8)y^3\\
&+(36u^4+62u^3+96u^2+54u+24)y^2\\
&+(32u^6+96u^5+228u^4+270u^3+240u^2+108u+32)y\\
&+16u^7+96u^6+212u^5+316u^4+292u^3+192u^2+72u+16.
\end{align*}
All coefficients in this expansion are positive. Hence $\Pi>0$ and $C>0$. Therefore, $\frac{d\nu}{dl}<0$, as required.
\end{proof}

We now compare the two candidate values in Lemma \ref{lem:cl1}.

\begin{lem}\label{pro:5.4}
Let $l\geq2$ and $r>(2l-1)(2l-2)$ be integers. Set
\begin{equation}\label{eq:Bdef}
B(r,l)= 2r^3+(3+12l-12l^2)r^2+(8l^4-16l^3-8l^2+18l)r
 +(8l^4-24l^3+16l^2).
\end{equation}
Then $\nu(r,l)<\xi(r,l)$ if and only if $B(r,l)<0$. Moreover, if $r\geq6l^2$, then $B(r,l)>0$ and consequently $\nu(r,l)\geq \xi(r,l)$.
\end{lem}

\begin{proof}
In configuration~\textup{(ii)} of Lemma \ref{lem:cl1}, the matrix $Q_1$ has two real eigenvalues, of which the larger one is $\nu(r,l)$. Therefore, $\nu(r,l)<\xi(r,l)$ if and only if both eigenvalues of $Q_1$ are less than $\xi(r,l)$, or equivalently  $\det(\xi(r,l)I-Q_1)>0$ and $\operatorname{tr}(Q_1)<2\xi(r,l)$.

Since $2\xi(r,l)-\operatorname{tr}(Q_1)=\frac{r^2-(l-2)r+l}{r-l+2}=\frac{r(r-l+2)+l}{r-l+2}>0$,we always have $\operatorname{tr}(Q_1)<2\xi(r,l)$. Moreover, expanding the $2\times2$ determinant and collecting terms gives $\det(\xi(r,l)I-Q_1)=-\frac{B(r,l)}{4(2l-1)(r-l+2)^2}$. Since the denominator is positive, $\det(\xi(r,l)I-Q_1)>0$ if and only if $B(r,l)<0$. This proves the first assertion.

We now prove the second assertion. Direct substitution gives $B(6l^2,l)=4l^2\bigl(12l^4+84l^3+17l^2+21l+4\bigr)>0$. Write $r=6l^2+x$, where $x\geq0$. Then
\[
\frac{\partial B}{\partial r}(6l^2+x,l)
=2\bigl(40l^4+64l^3+14l^2+9l+(24l^2+12l+3)x+3x^2\bigr)>0.
\]
Thus $B$ is strictly increasing with $r$ for $r\geq6l^2$. Since $B(6l^2,l)>0$, we conclude that $B(r,l)>0$ for all $r\geq6l^2$.
\end{proof}

\begin{proof}[Proof of Theorem~\ref{thm:1.3}]
Let $t\geq 2$ and $r>(2t-1)(2t-2)$ be integers, and let $G$ be a connected $r$-regular graph with $\kappa(G)=2l-1\leq2t-1$.
Let $S$ be a minimum vertex cut of $G$, and write $|S|=\kappa(G)=2l-1\leq2t-1$. The case $l=1$ was proved at the beginning of this section, so assume that $2\leq l\leq t$.

By the reduction preceding~\eqref{eq:odd-compact}, Lemma \ref{lem:cl1} and~\eqref{eq:odd-interlace}, we have $\lambda_2(G)\geq\lambda_1(Q_1)\geq\min\bigl\{\xi(r,l),\nu(r,l)\bigr\}$. Since $r>(2t-1)(2t-2)$, Lemma \ref{lem:5.3} applies. As $l\leq t$,
\[
\min\bigl\{\xi(r,l),\nu(r,l)\bigr\}\geq\min_{2\leq l'\leq t}\min\bigl\{\xi(r,l'),\nu(r,l')\bigr\}=\min\bigl\{\xi(r,t),\nu(r,t)\bigr\}.
\]
Thus~\eqref{eq:odd-bound} holds. 

Now suppose that $r\geq6t^2$. Then Lemma~\ref{pro:5.4} gives $B(r,t)>0$ and hence $\nu(r,t)\geq \xi(r,t)$. Therefore, the bound \eqref{eq:odd-bound} reduces to~\eqref{eq:odd-bounds}.
\end{proof}

\medskip
\noindent \textbf{Acknowledgements}~~The first author was supported by China Scholarship Council and Autonomous Region Graduate Research Innovation Project (No. XJ2026G012).

\end{document}